\documentclass[11pt,leqno]{article}
\usepackage{graphicx, amsfonts, amsthm, amsxtra, amssymb, verbatim, makeidx}
\usepackage{subeqnarray, relsize}
\usepackage[mathscr]{euscript}
\usepackage{hyperref}
\makeindex
\makeglossary
\usepackage{amsmath,amssymb, amsthm}
\usepackage[utf8]{inputenc}
\usepackage[activeacute, english]{babel}
\usepackage{wrapfig}
\usepackage{amssymb, amsmath, amsthm} 
\usepackage{graphicx}
\usepackage{color}
\usepackage{hyperref}
\usepackage{amssymb}
\usepackage{url}
\usepackage{pdfpages}
\usepackage{fancyhdr}
\usepackage{hyperref}
\usepackage{subfig}
\usepackage{mathrsfs}
\usepackage{float}
\usepackage{pgf,tikz}
\usetikzlibrary{arrows}
\usetikzlibrary[patterns]

\DeclareMathOperator{\tr}{tr}
\DeclareMathOperator{\pn}{{\sf p}}

\newcommand{\CH}{\text{CH}}
\newcommand{\res}{\text{res}}
\newcommand{\soc}{\text{soc}}
\newcommand{\NL}{\text{NL}}
\newcommand{\N}{\mathbb{N}}
\newcommand{\Z}{\mathbb{Z}}
\newcommand{\C}{\mathbb{C}}
\newcommand{\Q}{\mathbb{Q}}
\newcommand{\R}{\mathbb{R}}
\newcommand{\K}{\mathbb{K}}
\renewcommand{\P}{\mathbb{P}}
\newcommand{\T}{\mathbb{T}}
\newcommand{\U}{\mathcal{U}}

\renewcommand{\Im}{\mathrm{I}\text{m}}

\newcommand{\dR}{\text{dR}}

\newcommand{\prim}{\text{prim}}
\newcommand{\LT}{\text{LT}}
\newcommand{\p}[3]{\left #1 #3 \right #2}

\newtheorem{thm}{Theorem}[section]
\newtheorem{cor}[thm]{Corollary}
\newtheorem{prop}[thm]{Proposition}
\theoremstyle{definition}

\newtheorem{rmk}{Remark}[section]
\newtheorem{ex}{Example}[section]
\newtheorem{dfn}{Definition}[section]
\renewenvironment{proof}{{\bfseries \noindent Proof} }{ \qed \\}

\newcount\colveccount 
\newcommand*\colvec[1]{ 
    \global\colveccount#1
    \begin{pmatrix} 
    \colvecnext 
}

\def\colvecnext#1{
    #1
    \global\advance\colveccount-1 
    \ifnum\colveccount>0 
        \\
        \expandafter\colvecnext 
    \else 
        \end{pmatrix}
    \fi
}
\begin{document}

%\include{chapters/notations}

%------Chapter on algebraic cycles-----------------------
\def\gru{\mu} %the group of d-th roots of unity 
\def\pg{{ \sf S}}               %permutation group
\def\TS{\mathlarger{\bf T}}                %Tangent space
\def\NB{{\mathlarger{\bf N}}}
\def\group{{\sf G}}
\def\NLL{{\rm NL}}   %Noether-Lefschetz locus

\def\plc{{ Z_\infty}}    %a section of X by a projective space of dimension n/2+1
\def\pola{{u}}      %polarization
\newcommand\licy[1]{{\mathbb P}^{#1}} %linear projective cycle
\newcommand\aoc[1]{Z^{#1}}     % Aoki-Shioda cycle
\def\HL{{\rm HL}}     %Hodge locus
\def\NLL{{\rm NL}}   %Noether-Lefschetz locus

%----------------General Math Notations-------------------------------------
\def\Z{\mathbb{Z}}                   %Integer  numbers
\def\Q{\mathbb{Q}}                   %Rational  numbers
\def\C{\mathbb{C}}                   %Complex numbers
\def\N{\mathbb{N}}                   %natural numbers
\def\uhp{{\mathbb H}}                %upper half plane
\def\A{\mathbb{A}}                   %affine space C^n
\def\dR{{\rm dR}}                    %The subindex dR standing for de Rham cohomology.
\def\F{{\cal F}}                     %A foliation
\def\Sp{{\rm Sp}}                    %Symplectic group
\def\Gm{\mathbb{G}_m}                 %The multiplicative group
\def\Ga{\mathbb{G}_a}                 %The additive  group
\def\Tr{{\rm Tr}}                      %Trace map in Algebraic de Rham cohomology
\def\tr{{{\mathsf t}{\mathsf r}}}                 %Transposition of matrices
\def\spec{{\rm Spec}}            %The spectrume
\def\ker{{\rm ker}}              %kernel
\def\GL{{\rm GL}}                %The liner group
\def\ker{{\rm ker}}              %kernel
\def\coker{{\rm coker}}          %cokernel
\def\im{{\rm Im}}               %Image
\def\coim{{\rm Coim}}            %coimage
\def\p{{\sf  p}}
\def\U{{\cal U}}   %covering

\def\weig{{\nu}}
\def\r{{ r}}                       %dimension of the moduli space of hypersurfaces.
%----------------Gauss-Manin Connection in disguise---------------------
\def\k{{\sf k}}                     %Arbitrary field
\def\ring{{\sf R}}                   %A ring
\def\X{{\sf X}}                      %families of varieties
\def\Ua{{   L}}                      %families of affine varieties
\def\T{{\sf T}}                      %Moduli of enhanced  varieties
\def\asone{{\sf A}}                  %affine space A^{n+1} defined over a ring

\def\Ts{{\sf S}}
\def\cmv{{\sf M}}                    %Classical moduli of varieties
\def\BG{{\sf G}}                       %Borel Algebraic Group
\def\podu{{\sf pd}}                   %Poincare dual
\def\ped{{\sf U}}                    %Period Domain
\def\per{{\bf  P}}                   %period matrix or fundamental system of GMC
\def\gm{{  A}}                    %Gauss-Manin connection
\def\gma{{\sf  B}}                   %Gauss-Manin connection
\def\ben{{\sf b}}                    %Betti number

\def\Rav{{\mathfrak M }}                     % Space of modular vector fields
\def\Ram{{\mathfrak C}}                     % Space of constant vector fields
\def\Rap{{\mathfrak G}}                     % Space of vector fields arising from group  action. 

\def\mov{{\sf  m}}                    %Fixed dimension of the cohomology
\def\Yuk{{\sf C}}                     %Yukawa and generalizations 
\def\Ra{{\sf R}}                      %Ramanujan type vector field
\def\hn{{ h}}                         %Hodge numbers
\def\cpe{{\sf C}}                     %Constant periods
\def\g{{\sf g}}                       %An element of the Borel group
\def\t{{\sf t}}                       %An element \T
\def\pedo{{\sf  \Pi}}                  %Period domain before discrete group action

\def\Der{{\rm Der}}                   %Derivations
\def\MMF{{\sf MF}}                    %Moduli of modular foliations
\def\codim{{\rm codim}}                %codimension
\def\dim{{\rm    dim}}                %dimension
\def\Lie{{\rm Lie}}                   %Lie algebra of a group.

\def\u{{\sf u}}                       %An element of the period domain

\def\imh{{  \Psi}}                 %intersection matrix in homology
\def\imc{{  \Phi }}                  %intersection matrix in cohomology
\def\stab{{\rm Stab }}               %stablizer 
\def\Vec{{\rm Vec}}                 %space of vector fields

\def\Fg{{\sf F}}     %Genus g topological partition function
\def\hol{{\rm hol}}  %holomorphic
\def\non{{\rm non}}  %non-holomorphic
\def\alg{{\rm alg}}  %algebraic
\def\tra{{\rm tra}}  %transcendental

\def\bcov{{\rm \O_\T}}       %The ring of modular-type functions

\def\leaves{{\cal L}}        %space of leaves  

\def\cat{{\cal A}}              %category
\def\im{{\rm Im}}               %Image

%%%%%%%%%%%%%%%%%%%%%%%%%%%%%%%%%%%%%%%%%%%%%%%%%%%%%%%%%%%%%%%%%
\def\pn{{\sf p}}              %Periods of Hodge cycles
\def\Pic{{\rm Pic}}           %Picard group 
\def\free{{\rm free}}         %Free part
\def \NS{{\rm NS}}    %Neron-Severi group 
\def\tor{{\rm tor}}
\def\codmod{{\xi}}    %Codimension of the Hodge loci

%---------------OLD Notation-------------------
\def\GM{{\rm GM}}

\def\perr{{\sf q}}        %period matrix.....
\def\perdo{{\cal K}}   %period domain
\def\sfl{{\mathrm F}} %Space of filtrations
\def\sp{{\mathbb S}}  %Sphere

\newcommand\diff[1]{\frac{d #1}{dz}} %Differential operator
\def\End{{\rm End}}              %Endomorphism group

\def\sing{{\rm Sing}}            %The set of singularities
\def\cha{{\rm char}}             %Charracteristic
\def\Gal{{\rm Gal}}              %The Galois group
\def\jacob{{\rm jacob}}          %the Jacobian ideal
\def\tjurina{{\rm tjurina}}      %the tjurina ideal
\newcommand\Pn[1]{\mathbb{P}^{#1}}   %Projective space of dimension #1
\def\P{\mathbb{P}}
\def\Ff{\mathbb{F}}                  %Finite field

\def\O{{\cal O}}                     %ring of integers of a number field

\def\ring{{\mathsf R}}                         %A ring
\def\R{\mathbb{R}}                   %real numbers

\newcommand\ep[1]{e^{\frac{2\pi i}{#1}}}% unipotent numbers
\newcommand\HH[2]{H^{#2}(#1)}        %Hodge structures
\def\Mat{{\rm Mat}}              %Matrices
\newcommand{\mat}[4]{
     \begin{pmatrix}
            #1 & #2 \\
            #3 & #4
       \end{pmatrix}
    }                                %two by two matrices
\newcommand{\matt}[2]{
     \begin{pmatrix}                 % one by two matrix
            #1   \\
            #2
       \end{pmatrix}
    }
\def\cl{{\rm cl}}                %Chern class

\def\hc{{\mathsf H}}                 %The set of Hodge cycles.
\def\Hb{{\cal H}}                    %Hodge bundle
\def\pese{{\sf P}}                  %Period set

\def\PP{\tilde{\cal P}}              %the period domain/ discrete group
\def\K{{\mathbb K}}                  %Field representing R or C

\def\M{{\cal M}}
\def\RR{{\cal R}}
\newcommand\Hi[1]{\mathbb{P}^{#1}_\infty}%the hyperplane at infinity
\def\pt{\mathbb{C}[t]}               %Polynomials in t
\def\gr{{\rm Gr}}                %graded pieces
\def\Im{{\rm Im}}                %imaginary
\def\Re{{\rm Re}}                %Real
\def\depth{{\rm depth}}
\newcommand\SL[2]{{\rm SL}(#1, #2)}    %SL(2,Z)
\newcommand\PSL[2]{{\rm PSL}(#1, #2)}  %PSL(2,Z)
\def\Resi{{\rm Resi}}              %Residue

\def\L{{\cal L}}                     %The moduli of polarized lattices in a
                                     %fixed vector spaces.
\def\Aut{{\rm Aut}}              %Automorphism group of a vectorspace
\def\any{R}                          %Any subring of the field of complex
                                     %numbers.
\newcommand\ovl[1]{\overline{#1}}    %Conjugation of #1.

\newcommand\mf[2]{{M}^{#1}_{#2}}     %New modular functions
\newcommand\mfn[2]{{\tilde M}^{#1}_{#2}}     %New modular functions

\newcommand\bn[2]{\binom{#1}{#2}}    %Binomial
\def\ja{{\rm j}}                 %j of a two by two matrix
\def\Sc{\mathsf{S}}                  %Simple cycles
\newcommand\es[1]{g_{#1}}            %Eisenstein series
\newcommand\V{{\mathsf V}}           %Milnor vector space
\newcommand\WW{{\mathsf W}}          %Similar to Milnor vector space
\newcommand\Ss{{\cal O}}             %Structural sheaf
\def\rank{{\rm rank}}                %rank of a module
\def\Dif{{\cal D}}                   %Differentials
\def\gcd{{\rm gcd}}                  %greatest common divisor
\def\zedi{{\rm ZD}}                  %zero divisors of a module
\def\BM{{\mathsf H}}                 %Brieskorn module
\def\plf{{\sf pl}}                             %Picard-Lefschetz formula
\def\sgn{{\rm sgn}}                      %sign
\def\diag{{\rm diag}}                   %diagonal matrix
\def\hodge{{\rm Hodge}}
\def\HF{{ F}}                                %The hodge filtration of the brieskon module
\def\WF{{ W}}                               %The weight filtration of the brieskon module
\def\HV{{\sf HV}}                                %humbert variety
\def\pol{{\rm pole}}                               %pole divisor
\def\bafi{{\sf r}}
\def\id{{\rm id}}                               %identity
\def\gms{{\sf M}}                           %Gauss-Manin system
\def\Iso{{\rm Iso}}                           %Gauss-Manin system

\def\hl{{\rm L}}    %holomorphic limit
\def\imF{{\rm F}}
\def\imG{{\rm G}}

\begin{center}
{\LARGE\bf  Small codimension components of the Hodge locus containing the Fermat variety
}
%\footnote{ 
%Math. classification: Primary 14D07, 13H10; Secondary 13P10, 14C25
%\\
%Keywords: Hodge locus, Artinian Gorenstein algebra, Hodge cycle, periods. 
%}
\\
\vspace{.25in} {\large {\sc Roberto  Villaflor Loyola}}\footnote{
Instituto de Matem\'atica Pura e Aplicada, IMPA, Estrada Dona Castorina, 110, 22460-320, Rio de Janeiro, RJ, Brazil,
{\tt rvilla@impa.br}}
\end{center}

% \tableofcontents

\def\pn{{\sf p}}
\def\rootsG{{\sf G}}
\def\NLL{{\rm NL}}   %Noether-Lefschetz locus

\begin{abstract}
We characterize the smallest codimension components of the Hodge locus of smooth degree $d$ hypersurfaces of the projective space $\P^{n+1}$ of even dimension $n$, passing through the Fermat variety (with $d\neq 3,4,6$). They correspond to the locus of hypersurfaces containing a linear algebraic cycle of dimension $\frac{n}{2}$. Furthermore, we prove that among all the local Hodge loci associated to a non-linear cycle passing through Fermat, the ones associated to a complete intersection cycle of type $(1,1,\ldots,1,2)$ attain the minimal possible codimension of their Zariski tangent spaces. This answers a conjecture of Movasati, and generalizes a result of Voisin about the first gap between the codimension of the  components of the Noether-Lefschetz locus to arbitrary dimension, provided that they contain the Fermat variety.  
\end{abstract}

\section{Introduction}
\label{intro}
Let $\pi:X\rightarrow T$ be the family of all smooth degree $d$ hypersurfaces of $\P^{n+1}$, of even dimension $n$. Denote by  $\HL_{n,d}\subseteq T$ the set of parameters $t\in T$ corresponding to hypersurfaces containing \textit{non-trivial} Hodge cycles (with non-trivial primitive part), i.e. such that $H^{\frac{n}{2},\frac{n}{2}}(X_t,\Z)_\prim\neq 0$.
We know after \cite{gri83III} and \cite{cadeka} that for $d\ge 2+\frac{4}{n}$, $\HL_{n,d}$ is a countable union of properly contained algebraic subvarieties of $T$ called the \textit{Hodge locus}. This locus was first considered by Grothendieck \cite{gro66} in order to study the Hodge conjecture on families, the so called variational Hodge conjecture. This conjecture claims that the algebraicity of a Hodge cycle spreads along all flat deformations which remain Hodge. In terms of the Hodge locus, we know that at every point $t_0\in \HL_{n,d}$, every component of the germ of analytic variety $(\HL_{n,d},t_0)$ corresponds to the space of parameters $t\in (T,t_0)$ where a section $\lambda$ of the local system $R^n\pi_*\Z$ is a non-trivial Hodge cycle, i.e. $\lambda(t)\in H^{\frac{n}{2},\frac{n}{2}}(X_t,\Z)$ and $\lambda(t)_\prim\neq 0$. We will denote this \textit{local Hodge locus} by $V_\lambda$ (it has a natural structure of analytic scheme, that might be non-reduced). In this setting we can rephrase the variational Hodge conjecture as follows: $\lambda(t_0)_\prim\in H^{\frac{n}{2},\frac{n}{2}}(X_{t_0},\Z)_\prim$ is an algebraic class (i.e. $\lambda(t_0)_\prim=[Z]_\prim$ is the primitive part of the cohomology class of an algebraic cycle $Z\in \CH^\frac{n}{2}(X_{t_0})$) if and only if $\lambda(t)_\prim$ is an algebraic class for all $t\in V_\lambda$. Therefore, the variational Hodge conjecture can be regarded as a question about the germs of analytic subvariety of the locus of Hodge cycles over $\HL_{n,d}$.

For $n=2$, the study of the Hodge locus is a classical subject in algebraic geometry, since it corresponds to the \textit{Noether-Lefschetz locus} $\NL_d$ which parametrizes surfaces of degree $d\ge 4$ contained in $\P^3$ with Picard number bigger than 1 (i.e. those surfaces not satisfying the classical Noether-Lefschetz theorem). The development of the infinitesimal variations of Hodge structures \cite{gri83III} provided the necessary tool to study the first order approximations of the local Hodge loci $V_\lambda$. As a consequence, a series of results were discovered about the components of the Noether-Lefschetz locus until the beginning of the nineties. On one hand, just looking at the local Hodge loci we see that the Noether-Lefschetz locus is locally defined by $h^{2,0}={d-1\choose 3}$ equations, providing us with an upper bound for the codimension of its components. The components of codimension exactly $h^{2,0}$ are called \textit{general components} while the other components are called \textit{special components}. It was proved by Ciliberto, Harris and Miranda \cite{CHM88} that the Noether-Lefschetz locus has infinitely many general components which are dense in $T$ with respect to the analytic topology. On the other hand, the lower bound for the codimension of the special components is a subtle result due to Green \cite{greenkoszul2,green1988}, corresponding to $d-3$ which is attained by the locus of surfaces containing a line. Harris conjectured that the Noether-Lefschetz locus must have only finitely many special components, leading Voisin to an intensive study of these components. In \cite{voisin1988}, Voisin (and independently Green \cite{green1989}) proved that the unique component of the Noether-Lefschetz locus of minimal codimension corresponds to the locus of surfaces containing a line. Later Voisin \cite{voisin89} proved that there are no components of codimension bigger than $d-3$ and less than $2d-7$, and the unique component of codimension $2d-7$ corresponds to the locus of surfaces containing a conic. This result implied Harris conjecture for $d=5$. The conjecture for $d=6,7$ was also proved by Voisin \cite{voisin90}. Finally, Voisin \cite{voisin1991} disproved Harris conjecture for a high (not explicitly computed) degree. 

The study of the classical Noether-Lefschetz locus has evolved in several directions. One direction has been the use of the infinitesimal variations of Hodge structures to study the components of other families of surfaces. For instance complete intersection surfaces \cite{kim1991}, Jacobian elliptic surfaces \cite{kloosterman2007}, and surfaces inside complete normal toric threefolds \cite{bruzzo2017existence,bruzzo2018noether}.   

On the other hand, in the direction of higher dimensional hypersurfaces of $\P^{n+1}$ the knowledge about the Hodge locus is much more limited. For instance the variational Hodge conjecture has been verified for some restricted types of algebraic cycles (see \cite{Bloch1972} for semi-regular local complete intersection algebraic cycles, \cite{Otwinowska2003} and \cite{Dan14} for complete intersection algebraic cycles, \cite{MV} and \cite{VillaPCIAC} for some non-complete intersection algebraic cycles). While global results about the Hodge locus, such as the sharp bounds for the codimension of its components, and the characterization of the special components, remain open. 

The best known result about the sharp bounds and characterization of special components of the Hodge locus is due to Otwinowska \cite{Otw02}.
As in Voisin's proof for the classical Noether-Lefschetz locus case, Otwinowska uses the algebraic description of the infinitesimal variations of Hodge structures in terms of the product map in the Jacobian ring. However, due to the increment of the complexity of these rings in higher dimensions, she develops deep and technical results to study the Hilbert function of Artin Gorenstein algebras. As a result, Otwinowska manage to generalize Voisin's results to higher dimensions but only asymptotically \cite[Theorem3]{Otw02}. More precisely, she proves that for $d\gg n$, every component $\Sigma$ of $\HL_{n,d}$ satisfies
\begin{equation}
\label{cota}
\codim_T \ \Sigma\ge {\frac{n}{2}+d\choose d}-(\frac{n}{2}+1)^2,
\end{equation}
with equality if and only if $\Sigma$ is the locus of hypersurfaces containing a linear subvariety of dimension $\frac{n}{2}$.  

More recently, Movasati \cite[Theorem 2]{GMCD-NL} considered the same problem restricted to those components passing through the Fermat variety. He proved that all such components satisfy \eqref{cota} for all degrees $d\ge 2+\frac{4}{n}$. Note that since the Fermat variety contains a huge amount of algebraic cycles, there are lots of local components passing through it, and furthermore in the case of surfaces, all special components of small codimension pass trough the Fermat surface. Thus Otwinowska and Movasati's results provide evidences that \eqref{cota} should be the sharp lower bound for the components of the Hodge locus. In order to get the mentioned bound, Voisin, Green, Otwinowska and Movasati, they all found a bound for the Zariski tangent space of all local Hodge loci of the form $V_\lambda$ parametrizing some local component of the Hodge locus. In particular Movasati proved that if $0\in\HL_{n,d}$ denotes the Fermat variety, then
\begin{equation}
\label{cotatangent}
\codim_{T_0T} \ T_0V_\lambda\ge {\frac{n}{2}+d\choose d}-(\frac{n}{2}+1)^2,
\end{equation}
for all local Hodge loci $V_\lambda$ passing trough the Fermat variety. In fact he proves this bound without assuming $\lambda$ is an integral class and so the bound above holds for any $\lambda\in H^{\frac{n}{2},\frac{n}{2}}(X)_\prim$ (we prove this fact again in Proposition \ref{prop6}). And so all the non-reduced local Hodge loci passing trough Fermat (and the ones which are singular at Fermat) have strictly bigger codimension. After this result, and in view of Voisin/Green and Otwinowska's results, Movasati \cite[Conjecture 18.8]{ho13} conjectures that the only local Hodge loci attaining the equality in \eqref{cota} are those of the form $V_{\lambda}$ with $\lambda(0)_\prim=a[\P^\frac{n}{2}]_\prim$ for some $\P^\frac{n}{2}\subseteq X_0$ and some $a\in\Q^\times$. In this article we prove Movasati's conjecture.

\begin{thm}
\label{thm1} Let $n$ be an even number and $d$ be such that $\zeta_d+\zeta_d^{-1}\notin\Q$ (i.e. $d\neq 1,2,3,4,6$). Consider a component $\Sigma$ of $\HL_{n,d}$ passing through the \textit{Fermat variety} $0\in \HL_{n,d}$ given by $X_0=\{F=0\}\subseteq\P^{n+1}$ and $F=x_0^d+\cdots+x_{n+1}^d$. The equality in \eqref{cota} holds if and only if 
\begin{equation}
\label{complin}
\Sigma=\left\{t\in T: X_t \text{ contains a linear subvariety of dimension }\frac{n}{2}\right\}. 
\end{equation}
In fact, for any polydisc $0\in\Delta\subseteq T$ and for any $\lambda\in \Gamma(\Delta,R^n\pi_*\Z)$ such that $0\in V_\lambda$,
\begin{equation}
\label{igualdad1}
\codim_{T_0T} \ T_0V_\lambda= {\frac{n}{2}+d\choose d}-(\frac{n}{2}+1)^2
\end{equation}
if and only if $\lambda(0)_\prim=a[\P^\frac{n}{2}]_\prim$ for some $\P^\frac{n}{2}\subseteq X_0$ and some $a\in\Q^\times$.
\end{thm}

We remark that the subtle point in Voisin/Green and Otwoinowska's results is to prove that the bound \eqref{cota} holds at a generic point of the Hodge locus. Once this is established, the characterization of the components attaining this bound follows from a simple geometric argument. To give an idea of how this argument works we will recall the definition of Artinian Gorenstein algebra associated to a Hodge cycle of a smooth degree $d$ hypersurface of $\P^{n+1}$. This algebra was introduced by Voisin \cite{voisin89} for the case of surfaces, and by Otwinowska \cite{Otwinowska2003} for arbitrary dimension hypersurfaces. Let $X=\{F=0\}\subseteq\P^{n+1}$ be a smooth degree $d$ hypersurface of even dimension $n$. By Griffiths work \cite{gr69} we can use the residue map to identify pieces of the Jacobian ring
$$
R^F:=\C[x_0,\ldots,x_{n+1}]/J^F \ , \hspace{1cm}J^F:=\left\langle \frac{\partial F}{\partial x_0},\ldots,\frac{\partial F}{\partial x_{n+1}}\right\rangle,
$$
with the $(p,n-p)$ subgroups of $H^n_\dR(X)_\prim$ as follows
$$
R^F_{d(n-p+1)-n-2}\simeq F^pH^n_\dR(X)_\prim/F^{p+1}H^n_\dR(X)_\prim\simeq H^{p,n-p}(X)_\prim
$$
$$
P\mapsto \res\left(\frac{P\Omega}{F^{n-p+1}}\right)^{p,n-p}
$$
where $\Omega:=\sum_{i=0}^{n+1}(-1)^ix_idx_0\wedge\cdots\widehat{dx_i}\cdots\wedge dx_{n+1}$. In particular we can identify
$$
H^{\frac{n}{2},\frac{n}{2}}(X)_\prim\simeq R_\sigma \ , \ \text{ for }\sigma:=(d-2)(\frac{n}{2}+1).
$$
Let $\lambda\in H^{\frac{n}{2},\frac{n}{2}}(X,\Z)$ be a non-trivial Hodge cycle corresponding to a polynomial $P_\lambda\in R_\sigma$ such that
$$
\lambda_\prim=\res\left(\frac{P_\lambda\Omega}{F^{\frac{n}{2}+1}}\right)^{\frac{n}{2},\frac{n}{2}}.
$$
After \cite{CarlsonGriffiths1980} we can identify the Zariski tangent space of $V_\lambda$ with the inverse image in $\C[x_0,\ldots,x_{n+1}]_d$ of the kernel of the multiplication map $R^F_d\xrightarrow{\cdot P_\lambda}R^F_{\sigma+d}$ (see also \S \ref{sec2}). This motivates the definition of the \textit{Artinian Gorenstein algebra associated to $\lambda$} as 
$$
R^{F,\lambda}:=\C[x_0,\ldots,x_{n+1}]/J^{F,\lambda}$$
where
\begin{equation}
\label{defjflambda}
J^{F,\lambda}:=\{Q\in \C[x_0,\ldots,x_{n+1}]: P_\lambda\cdot Q=0\in R^F\}=(J^F:P_\lambda).
\end{equation}
Using elementary properties of this Artin Gorenstein algebra, Voisin and Otwinowska deduce that if \eqref{cota} is attained by some local Hodge locus $V_{\lambda}$ parametrizing the germ of analytic variety $(\Sigma,t)$ at some generic point $t\in\Sigma$, then 
\begin{equation}
\label{dimJFdelta1}
\dim_\C \ J^{F,\lambda(t)}_1=\frac{n}{2}+1.    
\end{equation}
Considering the induced map $L:\Sigma\rightarrow\mathbb{G}(\frac{n}{2},n+1)$ given by $L(t)=V(J^{F,\lambda(t)}_1)\subseteq\P^{n+1}$, Voisin \cite[\S 3]{voisin1988} and Otwinowska \cite[Theorem 3, Proposition 6]{Otw02} conclude (by a fibre dimension counting) that necessarily every $L(t)\subseteq X_t$ for all $t\in\Sigma$, thus they determine $\Sigma$.

The advantage of restricting ourselves to the Fermat point, is that bound \eqref{cotatangent} becomes considerably simpler since the Jacobian ideal $J^F=\langle x_0^{d-1},\ldots,x_{n+1}^{d-1}\rangle$ of the Fermat variety becomes monomial. Even though this is still a non-trivial bound since the Artin Gorenstein ideals $J^{F,\lambda(0)}$ are not monomial for an arbitrary Hodge cycle $\lambda(0)\in H^{\frac{n}{2},\frac{n}{2}}(X_0,\Z)$. Nevertheless we can handle them since they are quotient ideals of $J^F$ (see Proposition \ref{midef}). Once the bound is established at the Fermat point, the same methods of Voisin and Otwinowska allow us to deduce \eqref{dimJFdelta1} for all components of smallest codimension passing through Fermat. The subtle point in our context is to deduce that the linear subvariety $L(0)=V(J^{F,\lambda(0)})$ is in fact contained in the Fermat variety $X_0$. Since we just have \eqref{dimJFdelta1} only at one point (not even in a small neighbourhood of it) we cannot apply Voisin's geometric method. And to prove \eqref{dimJFdelta1} for all parameters of $V_\lambda$ is still a difficult commutative algebra problem (the main tools used by Voisin and Otwinowska, namely the theorems due to Macaulay, Grobmann and Green, are not enough in small degrees, see \cite{Otw02}). Despite the above we are able to show the following result.

\begin{thm}
\label{thm1.1}
Let $F=x_0^d+\cdots+x_{n+1}^d$ and $X=\{F=0\}$ be the Fermat variety of even dimension $n$ and degree $d$ such that $\zeta_d+\zeta_d^{-1}\notin\Q$ (i.e. $d\neq 1,2,3,4,6$). Let $\lambda\in H^{\frac{n}{2},\frac{n}{2}}(X,\Z)$ be a non-trivial Hodge cycle such that there exist $L_1,\ldots,L_{\frac{n}{2}+1}\in J^{F,\lambda}_1$ linearly independent. Then $$
\P^\frac{n}{2}:=\{L_1=\cdots=L_{\frac{n}{2}+1}=0\}\subseteq X.
$$
\end{thm}

Theorem \ref{thm1.1} together with Proposition \ref{prop6} and Proposition \ref{prop7} imply Theorem \ref{thm1}. In order to prove Theorem \ref{thm1.1} we exploit the fact that $\lambda$ is a rational class. This fact was never used in Voisin, Green, Otwinowska and Movasati's arguments. The main tools that allows us to use this information are \cite[Theorem 1.1, Proposition 6.1, Corollary 8.1]{VillaPCIAC} which allow us to compute the periods of $\lambda$ over algebraic cycles, in terms of the coefficients of the linear polynomials $L_1,\ldots,L_{\frac{n}{2}+1}$ defining the linear subspace $\P^\frac{n}{2}$. Computing these periods over all linear cycles of the Fermat variety we get enough constrains on the coefficients of $L_1,\ldots,L_{\frac{n}{2}+1}$ to conclude that $\P^\frac{n}{2}\subseteq X_0$ (see Proposition \ref{prop11}). The particular cases $d=3,4,6$ are more delicate since in these cases it is not enough to compute the periods of $\lambda$ over all linear cycles to conclude Theorem \ref{thm1.1} (even for $d=3,4$ where linear cycles generate all Hodge cycles \cite{Ran1980,Shioda1979,AMV17}). In fact, in \S\ref{sec5.5} we show infinite families of cohomology classes $\lambda\in F^\frac{n}{2}H^n_\dR(X)_\prim$ with all their periods over linear cycles having rational values and
$$
\dim_\C \ J^{F,\lambda}_1=\frac{n}{2}+1 \ \ \ \text{ but } \ \ \ V(J^{F,\lambda}_1)\not\subseteq X_0.
$$
We expect that these cohomology classes are not rational hence not Hodge cycles. We want to point out that Voisin's argument never uses the rationality of $\lambda$, and so it applies in general for any element of $H^{\frac{n}{2},\frac{n}{2}}(X_0)_\prim$. We discuss more about the relation between the equality \eqref{dimJFdelta1} and the rationality of $\lambda_\prim$ in \S \ref{sec6}.

Following the same spirit as Voisin's work \cite{voisin89}, we go further and determine the lower bound of the codimension of the components of the Hodge locus not satisfying \eqref{igualdad1} as follows.

\begin{thm}
\label{thm2}
Let $n$ be an even number, $d$ be a number such that $\zeta_d+\zeta_d^{-1}\notin\Q$ (i.e. $d\neq 1,2,3,4,6$), $0\in \HL_{n,d}$ be the Fermat variety and $0\in\Sigma$ be a component of the Hodge locus different from \eqref{complin}. Then
\begin{equation}
\codim_{T} \ \Sigma\ge{\frac{n}{2}+d\choose d}+{\frac{n}{2}+d-1\choose d-1}-(\frac{3n^2}{8}+\frac{9n}{4}+2).  
\end{equation}
In fact for $d\ge \text{max}\{4,2+\frac{6}{n}\}$, if $0\in\Delta\subseteq T$ is a polydisc and $\lambda\in \Gamma(\Delta,R^n\pi_*\Z)$ is such that $0\in V_\lambda$ but \eqref{igualdad1} does not hold, then 
\begin{equation}
\label{desteo2}
\codim_{T_0T} \ T_0V_\lambda\ge{\frac{n}{2}+d\choose d}+{\frac{n}{2}+d-1\choose d-1}-(\frac{3n^2}{8}+\frac{9n}{4}+2).  
\end{equation}
Furthermore, if the equality in \eqref{desteo2} is attained and $n\ge 4$, then there exists a complete intersection $Z\subseteq\P^{n+1}$ of type $(1,1,\ldots,1,2)$ such that $I(Z)\subseteq J^{F,\lambda}$. If in addition $d\ge 5$ and
\begin{equation}
\label{cicond}
F\in (J^{F,\lambda})^2,    
\end{equation}
then $Z\subseteq X_0$ and $\lambda(0)_\prim=a[Z]_\prim$ for some $a\in\Q^\times$. 
\end{thm}

It follows from Dan's result \cite{Dan14} that condition \eqref{cicond} is satisfied for every complete intersection algebraic cycle (see Remark \ref{condDan}). We do not know if this condition characterizes complete intersection algebraic cycles. 

The article is organized as follows: In \S \ref{sec2} we explain the algebraic translation of the Zariski tangent space of the Hodge loci $V_\lambda$ to the degree $d$ part of the Artin Gorenstein ideal $J^{F,\lambda}$. We also explain the main properties of this ideal, and its connection with the original definition in terms of periods due to Voisin \cite{voisin89} and Otwinowska \cite{Otwinowska2003}. Furthermore we give an elementary proof of Dan's result (Example \ref{dan}) which will be needed in several parts later in the article. In \S \ref{sec3} we establish the basic combinatorial bound of the Hilbert function of monomial ideals. We also explain how using a monomial ordering we can translate this bound into a bound of the Hilbert function of $J^{F,\lambda}$. We put all these facts together in \S \ref{sec4.1} to get Movasati's bound \eqref{cotatangent} and to characterize the equality in terms of $J^{F,\lambda}$, thus reducing the proof of Theorem \ref{thm1} to Theorem \ref{thm1.1}. The main part of the article is \S \ref{sec4.2} where we prove Theorem \ref{thm1.1}. Our proof relies heavily on explicit computations of the periods of $\lambda$ over all linear cycles inside the Fermat variety. Using the machinery developed in \cite{VillaPCIAC}, we can compute these periods once we find a closed formula for the polynomial $P_\lambda$ (this is done in Proposition \ref{prop10}). In \S \ref{sec5.5} we explain why our method do not cover the cases $d=3,4,6$. In \S \ref{sec5} we prove Theorem \ref{thm2} by describing the elements of $J^{F,\lambda}$ of degree one and two. Finally in \S \ref{sec6} we make some comments about how to deal with the cases $d=3,4,6$, and some further comments on the relation of the rationality of the class $\lambda$ and the continuity of the Hilbert function of $J^{F,\lambda}$ along the Hodge locus.

%Note that in the case of cubic hypersurfaces, for every complete intersection $Z\subseteq X$ of type $(1,1,\ldots,1,2)$ there exists a linear subvariety $\P^\frac{n}{2}\subseteq X$ such that $[Z]_\prim=-[\P^\frac{n}{2}]_\prim$. The gap given in \eqref{desteo2} for Fermat cubics has been verified experimentally by Movasati \cite[Table 5]{hosseincubic} and is attained by algebraic cycles supported in generalized cubic scrolls.

\bigskip

\noindent\textbf{Acknowledgements.} I am grateful to Prof. Remke Kloosterman for his comments and suggestions. Special thanks go to Prof. Hossein Movasati for his careful reading, suggestions and several useful conversations. I am also very grateful to the anonymous referee for his useful suggestions and corrections.

\section{Artinian Gorenstein algebra associated to a Hodge cycle}
\label{sec2}
The reduction of Theorem \ref{thm1} to Theorem \ref{thm1.1}, and the proof of Theorem \ref{thm2} are based on a translation to a purely algebraic problem. This is possible after giving the algebraic description of the Zariski tangent space of the local Hodge loci. In this section we will explain this description and its relation to the definition of the ideal $J^{F,\lambda}$ introduced in \eqref{defjflambda}. In spite definition \eqref{defjflambda} is the good one to work algebraically, this was not the original definition
due to Voisin \cite{voisin89} and Otwinowska \cite{Otwinowska2003} which was given in terms of periods. Let us recall this definition and connect it with \eqref{defjflambda}.

\begin{dfn}
Let $X=\{F=0\}\subseteq\P^{n+1}$ be a smooth degree $d$ hypersurface of even dimension $n$. Let $\lambda\in H^{\frac{n}{2},\frac{n}{2}}(X,\Z)$ be a non-trivial Hodge cycle. We define the homogeneous ideal
\begin{equation}
\label{hosdef}    
J^{F,\lambda}_i:=\left\{P\in \C[x_0,\ldots,x_{n+1}]_i: \int_X \res\left(\frac{PQ\Omega}{F^{\frac{n}{2}+1}}\right)\wedge\lambda=0, \forall Q\in \C[x_0,\ldots,x_{n+1}]_{\sigma-i}\right\}
\end{equation}
for every $i\ge 0$ and $\sigma:=(d-2)(\frac{n}{2}+1)$.
\end{dfn}

\begin{prop}
\label{tan}
Let $T\subseteq\C[x_0,\ldots,x_{n+1}]_d$ be the parameter space of smooth degree $d$ hypersurfaces of $\P^{n+1}$, of even dimension $n$. For $t\in T$, let $X_t=\{F=0\}\subseteq \P^{n+1}$ be the corresponding hypersurface. For every Hodge cycle $\lambda\in H^{\frac{n}{2},\frac{n}{2}}(X_t,\Z)$, we can compute the Zariski tangent space of its associated Hodge locus $V_\lambda$ as
$$
T_tV_\lambda=\left\{P\in \C[x_0,\ldots,x_{n+1}]_d: \int_{X}\res\left(\frac{PQ\Omega}{F^{\frac{n}{2}+1}}\right)\wedge\lambda=0, \forall Q\in \C[x_0,\ldots,x_{n+1}]_{d\frac{n}{2}-n-2}\right\}=J^{F,\lambda}_d.
$$
Where $\Omega=\sum_{i=0}^{n+1}(-1)^ix_idx_0\wedge\cdots\widehat{dx_i}\cdots\wedge dx_{n+1}$, and $\res: H^{n+1}_\dR(\P^{n+1}\setminus X_t)\rightarrow H^n_\dR(X_t)$ is the \textit{residue map}. Note that we have identified $T_tT\simeq \C[x_0,\ldots,x_{n+1}]_d$.
\end{prop}

\begin{proof}
This is a well-known result that follows from the description of the Zariski tangent space of the Hodge locus associated to a Hodge cycle by means of the infinitesimal variations of Hodge structures \cite[Lemma 5.16]{vo03} and the fact that the infinitesimal variations of Hodge structures corresponds to polynomial multiplication in the case of hypersurfaces \cite[Theorem 6.17]{vo03}.
\end{proof}

Before proving that
$R^{F,\lambda}:=\C[x_0,\ldots,x_{n+1}]/J^{F,\lambda}$ is an Artinian Gorenstein algebra, let us recall some basic facts about these algebras.

\begin{dfn}
A graded $\C$-algebra $R$ is \textit{Artinian Gorenstein} if there exist $\sigma\in \N$ such that 
\begin{itemize}
    \item[(i)] $R_e=0\text{ for all }e>\sigma$,
    \item[(ii)] $\dim_\C \ R_\sigma=1$,
    \item[(iii)] $\text{the multiplication map }R_i\times R_{\sigma-i}\rightarrow R_\sigma\text{ is a perfect pairing for all }i=0,\ldots,\sigma.$
\end{itemize}
The number $\sigma=:\soc(R)$ is the \textit{socle of $R$}. We say that an ideal $I\subseteq\C[x_0,\ldots,x_{n+1}]$ is \textit{Artinian Gorenstein of socle} $\sigma=:\soc(I)$ if the quotient ring $R=\C[x_0,\ldots,x_{n+1}]/I$ is Artinian Gorenstein of socle $\sigma$.
\end{dfn}

\begin{rmk}
An elementary property of Artinian Gorenstein ideals $I,J\subseteq\C[x_0,\ldots,x_{n+1}]$ is that 
$$
I=J \ \ \ \ \ \  \text{ if and only if } \ \ \ \ \ \ \soc(I)=\sigma=\soc(J)\text{ and }I_\sigma=J_\sigma.
$$
In particular if $I\subseteq J$ and $\soc(I)=\soc(J)$ then $I=J$. This property will be used several times along the present article without explicit mention.
\end{rmk}

\begin{prop}
\label{midef}
Let $X=\{F=0\}\subseteq\P^{n+1}$ be a smooth degree $d$ hypersurface of even dimension $n$, and $\lambda\in H^{\frac{n}{2},\frac{n}{2}}(X,\Z)$ be a non-trivial Hodge cycle. Then $R^{F,\lambda}$ is an Artinian Gorenstein algebra over $\C$ of socle $\sigma:=(d-2)(\frac{n}{2}+1)$. Furthermore if $J^F:=\langle \frac{\partial F}{\partial x_0},\ldots,\frac{\partial F}{\partial x_{n+1}}\rangle$ is the Jacobian ideal, then
\begin{equation}
\label{ideal}    
J^{F,\lambda}=(J^F:P_\lambda),
\end{equation}
for $P_\lambda\in \C[x_0,\ldots,x_{n+1}]_\sigma$ such that $\lambda_\prim=\res\left(\frac{P_\lambda\Omega}{F^{\frac{n}{2}+1}}\right)^{\frac{n}{2},\frac{n}{2}}$.
\end{prop}

\begin{proof}
It follows from Macaulay's theorem \cite[Theorem 6.19]{vo03} that the \textit{Jacobian ring} $R^F:=\C[x_0,\ldots,x_{n+1}]/J^F$ is an Artinian Gorenstein $\C$-algebra of socle $2\sigma$. Using this fact, it is elementary to show that if $P_\lambda\notin J^F$ then $(J^F:P_\lambda)$ is Artinian Gorenstein of socle $\sigma$. Thus, the proposition is reduced to prove \eqref{ideal}. Take $P\in \C[x_0,\ldots,x_{n+1}]_i$ and $Q\in \C[x_0,\ldots,x_{n+1}]_{\sigma-i}$ for $i\in\{0,1,\ldots,\sigma\}$. It follows from Griffiths basis theorem \cite[Theorem 6.5]{vo03} that $
\lambda_\prim=\res\left(\frac{P_\lambda\Omega}{F^{\frac{n}{2}+1}}\right)^{\frac{n}{2},\frac{n}{2}},
$
for some $P_\lambda\in\C[x_0,\ldots,x_{n+1}]_\sigma$. And it follows from a result due to Carlson and Griffiths \cite[Theorem 2]{CarlsonGriffiths1980} (see also Proposition \ref{prop8} for a more precise result) that 
\begin{equation}
\label{equivalence1}
\int_X \res\left(\frac{PQ\Omega}{F^{\frac{n}{2}+1}}\right)\wedge \res\left(\frac{P_\lambda\Omega}{F^{\frac{n}{2}+1}}\right)=0 \ \ \ \Longleftrightarrow  \ \ \ PQP_\lambda\in J^F.
\end{equation}
Therefore $J^{F,\lambda}_i=\{P\in \C[x_0,\ldots,x_{n+1}]_i: QPP_\lambda=0\in R^F_{2\sigma}, \forall Q\in R^F_{\sigma-i}\}$. By the perfect pairing condition on the Jacobian algebra it follows that $J^{F,\lambda}=(J^F:P_\lambda)$. Finally $P_\lambda\notin J^F$, since otherwise $\lambda_\prim=0\in H^{\frac{n}{2},\frac{n}{2}}(X)_\prim$ by \eqref{equivalence1}, contradicting our hypothesis.
\end{proof}

\begin{rmk}
The above proposition tells us how to relate the Artin Gorenstein ideal $J^{F,\lambda}$ associated to any Hodge cycle, to the Jacobian ideal $J^F$ of the corresponding hypersurface. Furthermore it tells us how to compute it once we know the polynomial $P_\lambda$ which determines the cohomological class of $\lambda_\prim$ in terms of Griffiths basis theorem. After \cite[Theorem 1.1]{VillaPCIAC} we can compute this polynomial explicitly for combinations of complete intersection algebraic cycles $\lambda=n_1[Z_1]+\cdots+n_k[Z_k]$ in terms of the defining equations of the complete intersection subvarieties $Z_1,\ldots,Z_k\subseteq X$. Essentially knowing $P_\lambda$ implies knowing the cohomological class $\lambda_\prim=\res\left(\frac{P_\lambda\Omega}{F^{\frac{n}{2}+1}}\right)^{\frac{n}{2},\frac{n}{2}}$, the Artin Gorenstein ideal $J^{F,\lambda}=(J^F:P_\lambda)$ and all the periods of $\lambda_\prim$ over $F^\frac{n}{2}H^n_\dR(X)$ \cite[Proposition 6.1]{VillaPCIAC}. The following corollary tells us that the Artin Gorenstein ideal $J^{F,\lambda}$ also determines uniquely $\lambda_\prim\in H^{\frac{n}{2},\frac{n}{2}}(X,\Q)_\prim$ modulo the action of $\Q^\times$.
\end{rmk}

\begin{cor}
\label{igualjotas}
Let $X=\{F=0\}\subseteq\P^{n+1}$ be a smooth degree $d$ hypersurface of even dimension $n$, and consider two Hodge cycles $\lambda_1,\lambda_2\in H^{\frac{n}{2},\frac{n}{2}}(X,\Z)$. Then
$$
J^{F,\lambda_1}=J^{F,\lambda_2} \ \ \Longleftrightarrow \ \ \exists c\in\Q^\times: (\lambda_1-c\cdot\lambda_2)_\prim=0.
$$
\end{cor}

\begin{proof}
Take some $x^\alpha\in\C[x_0,\ldots,x_{n+1}]_\sigma\setminus J^{F,\lambda_1}_\sigma$ (it must exist since $R^{F,\lambda_1}_\sigma\neq 0$). By definition \eqref{hosdef}, there exists a unique $c\in\C^\times$ such that 
$$
\int_{X} \res\left(\frac{x^\alpha\Omega}{F^{\frac{n}{2}+1}}\right)\wedge\lambda_1=c\cdot \int_{X}\res\left(\frac{x^\alpha\Omega}{F^{\frac{n}{2}+1}}\right)\wedge\lambda_2. 
$$
In consequence $x^\alpha\in J^{F,\lambda_1-c\cdot \lambda_2}$. On the other hand, since $J^{F,\lambda_1}=J^{F,\lambda_2}$ it follows that $J^{F,\lambda_1}\subset J^{F,\lambda_1-c\cdot\lambda_2}$, and this inclusion is proper. Thus $R^{F,\lambda_1-c\cdot\lambda_2}_\sigma=0$, then by Proposition \ref{midef} we have $(\lambda_1-c\cdot \lambda_2)_\prim=0$. Finally since $\lambda_1$ and $\lambda_2$ are integral classes $c\in\Q$.
\end{proof}

\begin{rmk}
Proposition \ref{tan} tells us that the degree $d$ piece of $J^{F,\lambda}$ encodes the information of the Zariski tangent space of the local Hodge locus $V_\lambda$. In other words $T_tV_{\lambda_1}=T_tV_{\lambda_2}$ if and only if $J^{F,\lambda_1}_d=J^{F,\lambda_2}_d$. Moreover, Corollary \ref{igualjotas} shows that the whole ideal $J^{F,\lambda}$ contains more information than just the first order approximation of $V_\lambda$, and in fact it determines $V_\lambda$. By this we mean that $V_{\lambda_1}=V_{\lambda_2}$ scheme-theoretically (hence for all higher order approximations) if and only if $J^{F,\lambda_1}=J^{F,\lambda_2}$. 
\end{rmk}

\begin{ex}
\label{dan}
It was proved by Dan \cite[Theorem 1.1]{Dan14} that for $d\ge 2+\frac{4}{n}$
$$
\Sigma_{(d_1,\ldots,d_{\frac{n}{2}+1})}:=\{t\in T: X_t\text{ contains a complete intersection of type }(d_1,\ldots,d_{\frac{n}{2}+1})\}
$$ 
is a component of $\HL_{n,d}$. This result is important for us and will be needed later. Because of its simplicity, we decided to reproduce Dan's argument (translated to our notation) as an example of the usefulness of considering the whole ideal $J^{F,\lambda}$ instead of just looking at its degree $d$ part (which corresponds to the Zariski tangent space of the associated Hodge locus $V_\lambda$). Consider for every $t\in \Sigma_{(d_1,\ldots,d_{\frac{n}{2}+1})}$ the Hodge cycle 
$$
\lambda=[Z]\in H^{\frac{n}{2},\frac{n}{2}}(X_t,\Z),
$$
where $Z$ is a complete intersection cycle of type $(d_1,\ldots,d_{\frac{n}{2}+1})$ given by 
$$
I(Z)=\langle f_1,\ldots,f_{\frac{n}{2}+1}\rangle,\hspace{3mm}\text{deg}(f_i)=d_i.
$$
If $X_t=\{F=0\}$, we can write $F=f_1g_1+\cdots+f_{\frac{n}{2}+1}g_{\frac{n}{2}+1}$. It follows from the definition \eqref{hosdef} of $J^{F,\lambda}$ that
$
I(Z)\subseteq J^{F,\lambda}.
$
On the other hand, if we define 
$$
Z_i:=\{f_1=\cdots=f_{i-1}=g_i=f_{i+1}=\cdots=f_{\frac{n}{2}+1}=0\},
$$
then $[Z]+[Z_i]=[X_t\cap\{f_1=\cdots =f_{i-1}=f_{i+1}=\cdots=f_{\frac{n}{2}+1}=0\}]$, and so $[Z_i]_\prim=-[Z]_\prim$. In consequence $I(Z_i)\subseteq J^{F,\lambda}$ for every $i=1,\ldots,\frac{n}{2}+1$, then 
$$
I:=\langle f_1,g_1,\ldots,f_{\frac{n}{2}+1},g_{\frac{n}{2}+1}\rangle\subseteq J^{F,\lambda}.
$$
By Macaulay's theorem \cite[Theorem 6.19]{vo03} it follows that $\soc(I)=\soc(J^{F,\lambda})$ and so $I=J^{F,\lambda}$. Therefore
$$
T_tV_\lambda=\langle f_1,g_1,\ldots,f_{\frac{n}{2}+1},g_{\frac{n}{2}+1}\rangle_d.
$$
Letting $S:=\C[x_0,\ldots,x_{n+1}]$, the tangent space of $S_{d_1}\times S_{d-d_1}\times\cdots\times S_{d_{\frac{n}{2}+1}}\times S_{d-d_{\frac{n}{2}+1}}$ at $(f_1,g_1,\ldots,f_{\frac{n}{2}+1},g_{\frac{n}{2}+1})$ goes to $\langle f_1,g_1,\ldots,f_{\frac{n}{2}+1},g_{\frac{n}{2}+1}\rangle_d$ via the map 
$$
\varphi(P_1,Q_1,\ldots,P_{\frac{n}{2}+1},Q_{\frac{n}{2}+1}):=P_1Q_1+\cdots+P_{\frac{n}{2}+1}Q_{\frac{n}{2}+1}.
$$
This implies that $V_\lambda$ is smooth and reduced. And furthermore if $\Lambda$ is the set of all complete intersections of type $(d_1,\ldots,d_{\frac{n}{2}+1})$ contained in $X_t$, then
$$
(\Sigma_{(d_1,\ldots,d_{\frac{n}{2}+1})},t)=\bigcup_{Z\in \Lambda}V_{[Z]}.
$$
Finally, to conclude that $\Sigma_{(d_1,\ldots,d_{\frac{n}{2}+1})}$ is in fact a component of $\HL_{n,d}$, it is enough to show that no local Hodge loci contains properly $V_{[Z]}$. In fact, if $V_{[Z]}\subseteq V_{\lambda'}$ then 
$$
J^{F,[Z]}_d=\langle f_1,g_1,\ldots,f_{\frac{n}{2}+1},g_{\frac{n}{2}+1}\rangle_d=T_tV_{[Z]}\subseteq T_tV_{\lambda'}=J^{F,\lambda'}_d.
$$
Since $\sigma-d_i\ge d-d_i$ it follows that for every monomial $x^\alpha$ of degree $\sigma-d_i$, $f_i\cdot x^\alpha\in J^{F,\lambda'}$ and so by the perfect pairing on $R^{F,\lambda'}$ we conclude that $f_i\in J^{F,\lambda'}$ for every $i=1,\ldots,\frac{n}{2}+1$. Analogously we prove that the $g_i$'s also belong to $J^{F,\lambda'}$ and so $J^{F,\lambda'}=J^{F,[Z]}$. By Corollary \ref{igualjotas} it follows that $V_{\lambda'}=V_{[Z]}$ as desired. 
\end{ex}

\begin{rmk}
\label{condDan}
Since $F=f_1g_1+\cdots+f_{\frac{n}{2}+1}g_{\frac{n}{2}+1}$ and $J^{F,\lambda}=\langle f_1,g_1,\ldots,f_{\frac{n}{2}+1},g_{\frac{n}{2}+1}\rangle$, we see that $F\in (J^{F,\lambda})^2$ for every complete intersection algebraic cycle $\lambda$.
\end{rmk}

\section{Preliminaries about monomials}
\label{sec3}
In this section, we recall some elementary facts about monomials and monomial ideals that will be needed later. These facts can be found in any standard reference of computational commutative algebra, we use \cite{CoxIVA} as reference. The purpose of this section is to fix our notations and to prove Proposition \ref{prophos} which is the main combinatorial fact we use to bound the Hilbert function of monomial ideals. Using monomial orderings we will reduce the analysis of the Hilbert function of the Artin Gorenstein ideal $J^{F,\lambda}$ to the case of the monomial ideals $\langle\LT(J^{F,\lambda})\rangle$ (see Proposition \ref{hilb}). Our strategy to prove Theorem \ref{thm1} and Theorem \ref{thm2} will be to characterize first the ideals $\langle\LT(J^{F,\lambda})\rangle$ and then, using Gr\"obner basis theory, characterize the ideals $J^{F,\lambda}$. 

\begin{dfn}
Consider the polynomial ring $\C[x_0,\ldots,x_{n+1}]$. Let $M$ be the set of all monomials in this ring. A \textit{monomial ordering} is a multiplicative linear order on $M$. The unique example of monomial order we will use is the \textit{lexicographic order}, that orders first the variables $x_{\sigma(0)}<\cdots<x_{\sigma(n+1)}$ for some permutation $\sigma\in \text{Perm}(0,1,\ldots,n+1)$, and then extends the order to all $M$ lexicographically.
\end{dfn}

\begin{dfn}
Consider over $\C[x_0,\ldots,x_{n+1}]$ a fixed monomial order. We define the \textit{leading term of $f\in \C[x_0,\ldots,x_{n+1}]$} by
$$
\text{LT}(f):=\text{max}\{x^\alpha\in M: \text{the monomial $x^\alpha$ appears in the expression of $f$}\}.
$$
Let $I\subseteq\C[x_0,\ldots,x_{n+1}]$ be an ideal. The \textit{leading terms ideal} is by definition
$$
\langle\text{LT}(I)\rangle:=\langle\{\text{LT}(f): f\in I\}\rangle.
$$
\end{dfn}

\begin{prop}
\label{hilb}
Let $I\subseteq\C[x_0,\ldots,x_{n+1}]$ be an ideal. Then the ideal $\langle\text{LT}(I)\rangle$ has the same Hilbert function as $I$.
\end{prop}

\begin{proof}
See \cite[Proposition 4, page 458]{CoxIVA}.
\end{proof}

\begin{rmk}
\label{comp}
In particular, in order to compute the dimension of $\C[x_0,\ldots,x_{n+1}]_d/I_d$ it is enough to compute the dimension of $\C[x_0,\ldots,x_{n+1}]_d/\langle\text{LT}(I)\rangle_d$, and since $\langle\text{LT}(I)\rangle$ is a monomial ideal
$$
\dim_\C\left(\frac{\C[x_0,\ldots,x_{n+1}]}{\langle\text{LT}(I)\rangle}\right)_d=\#\{x^\alpha\in\C[x_0,\ldots,x_{n+1}]\setminus\langle\text{LT}(I)\rangle: \text{deg}(x^\alpha)=d\}.
$$
\end{rmk}

\begin{dfn}
Consider over $\C[x_0,\ldots,x_{n+1}]$ a fixed monomial order. Let $I=\langle f_1,\ldots,f_k\rangle\subseteq\C[x_0,\ldots,x_{n+1}]$ be an ideal. We say that $f_1,\ldots,f_k$ is a \textit{Gr\"obner basis of $I$ (with respect to the monomial order)} if
$$
\langle \LT(I)\rangle=\langle \LT(f_1),\ldots,\LT(f_k)\rangle.
$$
\end{dfn}

\begin{prop}
\label{grob}
Consider over $\C[x_0,\ldots,x_{n+1}]$ a fixed monomial order. Let $I=\langle f_1,\ldots,f_k\rangle\subseteq\C[x_0,\ldots,x_{n+1}]$ be an ideal. The generators $f_1,\ldots,f_k\in \C[x_0,\ldots,x_{n+1}]$ are a Gr\"obner basis of $I$ if and only if the polynomials $g\in \C[x_0,\ldots,x_{n+1}]$ leaving zero residue after dividing it by $f_1,\ldots,f_k$ are exactly the $g\in I$.
\end{prop}

\begin{proof}
 See \cite[\S 6, Proposition 1 and Corollary 2, page 82]{CoxIVA}.
\end{proof}

We close this section with the following elementary proposition about monomials which was inspired by \cite[Proposition 8]{GMCD-NL}.

\begin{prop}
\label{prophos}
Let $x^\alpha\in \C[x_0,\ldots,x_{n+1}]$ be any monomial. For every $d>0$ let
$$
S^d_\alpha:=\{x^\beta\in \C[x_0,\ldots,x_{n+1}]_d: x^\beta|x^\alpha\}.
$$
Then, for any $i\neq j$ such that $0<\alpha_i\le \alpha_j$,
\begin{equation}
\label{tecdes}
\#S^d_{\alpha'}\le\#S^d_\alpha
\end{equation}
where $\alpha'_k:=\alpha_k$ for $k\neq i,j$, $\alpha'_i:=\alpha_i-1$ and $\alpha'_j:=\alpha_j+1$. Moreover \eqref{tecdes} is a strict inequality if $\alpha_j\le d\le \text{deg}(x^\alpha)-\alpha_i$. As a consequence, for every $x^\alpha|(x_0\cdots x_{n+1})^{d-2}$ of $\text{deg}(x^\alpha)=\sigma=(d-2)(\frac{n}{2}+1)$ we have
$$
\#S^d_\alpha\ge {\frac{n}{2}+d\choose d}-(\frac{n}{2}+1)^2,
$$
with equality if and only if 
\begin{equation}
\label{al1}
\#\{0\le i\le n+1:\alpha_i=d-2\}=\#\{0\le i\le n+1:\alpha_i=0\}=\frac{n}{2}+1,    
\end{equation}
i.e. up to some relabeling of the coordinates $\alpha=(0,\ldots,0,d-2,\ldots,d-2)$. Furthermore, for those $\alpha$ not satisfying \eqref{al1}, we have for $d\ge 4$ that
$$
\#S^d_\alpha\ge {\frac{n}{2}+d\choose d}+{\frac{n}{2}+d-1\choose d-1}-\left(\frac{3n^2}{8}+\frac{9n}{4}+2\right),
$$
with equality if and only if up to some relabeling of the coordinates $\alpha=(0,\ldots,0,1,d-3,d-2,\ldots,d-2)$.
\end{prop}

\begin{proof}
Consider the map $f: S^d_{\alpha'}\rightarrow S^d_\alpha$ given by
$$
f(x^{\beta'}):= \left\{
	\begin{array}{ll}
		   x^{\beta'}\cdot(x_i/x_j)  & \mbox{if } \beta'_j>0, \\
		x^{\beta'}\cdot(x_j/x_i)^{\beta'_i} & \mbox{if }\beta'_j=0. 
	\end{array}
\right.
$$
This map is clearly injective. For the case $\alpha_j\le d\le \text{deg}(x^\alpha)-\alpha_i$, $f$ is not surjective since there exist $x^\beta\in S^d_\alpha$ with $\beta_i=0$ and $\beta_j=\alpha_j$. These monomials do not belong to the image of $f$.
\end{proof}

\section{Reduction of Theorem \ref{thm1} to Theorem \ref{thm1.1}}
\label{sec4.1}
In this section we give a proof of Theorem \ref{thm1} assuming Theorem \ref{thm1.1}. In order to establish the bound \eqref{cotatangent} we use Proposition \ref{tan} to reduce ourselves to bounding the Hilbert function of $J^{F,\lambda}$. Following the strategy explained in \S \ref{sec3}, we reduce ourselves to study the Hilbert function of the monomial ideal $\langle\LT(J^{F,\lambda})\rangle$, which in turn can be bounded using Proposition \ref{prophos}. Putting all this together we obtain the following proposition.

\begin{prop}
\label{prop6}
Let $X=\{F=0\}\subseteq \P^{n+1}$ be a smooth hypersurface of even dimension $n$ and degree $d\ge 2+\frac{4}{n}$. Let $\lambda\in F^\frac{n}{2}H^n_\dR(X)_\prim\setminus F^{\frac{n}{2}+1}H^n_\dR(X)_\prim$ and $\sigma:=(d-2)(\frac{n}{2}+1)$. If for some monomial order there exists $x^\alpha\in \C[x_0,\ldots,x_{n+1}]_\sigma\setminus \langle \LT(J^{F,\lambda})\rangle_\sigma$ such that $x^\alpha|x_0^{d-2}\cdots x_{n+1}^{d-2}$, then 
\begin{equation}
\label{des1}    
\dim_\C \ R^{F,\lambda}_d\ge {\frac{n}{2}+d\choose d}-(\frac{n}{2}+1)^2.
\end{equation}
Furthermore if the equality holds in \eqref{des1} and $d\ge 2+\frac{6}{n}$, then there exist $L_1,\ldots,L_{\frac{n}{2}+1}\in J^{F,\lambda}_1$ linearly independent. 
\end{prop}

\begin{proof}
Under the hypothesis it follows from Proposition \ref{hilb} and Remark \ref{comp} that 
\begin{equation}
\label{dessalfa}
\dim_\C \ R^{F,\lambda}_d\ge \#S^d_\alpha.
\end{equation}
Thus \eqref{des1} follows directly from Proposition \ref{prophos}. Assuming that the equality holds for $d\ge 2+\frac{6}{n}$, again by Proposition \ref{prophos} 
\begin{equation}
\label{monmin}    
x^\alpha=x_{i_1}^{d-2}\cdots x_{i_{\frac{n}{2}+1}}^{d-2}
\end{equation}
for some $0\le i_1<\cdots<i_{\frac{n}{2}+1}\le n+1$. After relabeling the variables we can assume $x^\alpha=x_1^{d-2}x_3^{d-2}\cdots x_{n+1}^{d-2}$. Then for every $k\ge 0$
$$
S^k_\alpha\subseteq \{x^\beta\in\C[x_0,\ldots,x_{n+1}]_k\setminus \langle \LT(J^{F,\lambda})\rangle_k: x^\beta\text{ is a monomial}\},
$$ 
otherwise there would be a monomial $x^\beta|x^\alpha$ inside $\langle\LT(J^{F,\lambda})\rangle$ but this contradicts the fact $x^\alpha\notin\langle\LT(J^{F,\lambda})\rangle$. We claim both sets are equal for $k=1$, in other words we claim 
\begin{equation}
\label{igualmonom}
\langle \LT(J^{F,\lambda})\rangle_k\subseteq\langle x_0,x_2,\ldots,x_n,x_1^{d-1},x_3^{d-1},\ldots,x_{n+1}^{d-1}\rangle_k
\end{equation}
is an equality for $k=1$. By Remark \ref{comp} we know \eqref{igualmonom} is an equality for $k=d$, since both sides have the same codimension inside $\C[x_0,\ldots,x_{n+1}]_d$. For $k>d$, take any 
$$
x^\beta\in\langle x_0,x_2,\ldots,x_n,x_1^{d-1},x_3^{d-1},\ldots,x_{n+1}^{d-1}\rangle_k.
$$
Then for some $i\in\{0,\ldots,n+1\}$, $\beta_i\ge 1$ and $i$ is even, or $\beta_i\ge d-1$ and $i$ is odd. In any case we can produce some $x^\gamma|x^\beta$ such that $x^\gamma\in \langle x_0,x_2,\ldots,x_n,x_1^{d-1},x_3^{d-1},\ldots,x_{n+1}^{d-1}\rangle_d=\langle \LT(J^{F,\lambda})\rangle_d$ and so $x^\beta\in\langle \LT(J^{F,\lambda})\rangle_k$. Since $d\ge 2+\frac{6}{n}$ then $\sigma-1\ge d$ and so the equality follows for $k=1$ by duality. Therefore there exist $L_i\in J^{F,\lambda}_1$ such that $\LT(L_i)=x_{2i-2}$ for $i=1,\ldots,\frac{n}{2}+1$ as desired.
\end{proof}

\begin{rmk}
Note that the condition $d\ge 2+\frac{6}{n}$ cannot be improved since for $(n,d)=(2,4), (4,3)$ the equality in \eqref{des1} is attained by all local Hodge loci $V_\lambda$.
\end{rmk}

The next proposition essentially reduces the proof of Theorem \ref{thm1} to proving Theorem \ref{thm1.1}.

\begin{prop}
\label{prop7}
Let $X=\{F=0\}\subseteq \P^{n+1}$ be a smooth degree $d$ hypersurface of even dimension $n$, and $\lambda\in H^{\frac{n}{2},\frac{n}{2}}(X,\Z)$ be a non-trivial Hodge cycle. If there exist $L_1,\ldots,L_{\frac{n}{2}+1}\in J^{F,\lambda}_1$ linearly independent, such that 
$$
\P^\frac{n}{2}:=\{L_1=L_2=\cdots=L_{\frac{n}{2}+1}=0\}\subseteq X.
$$
Then there exists $c\in\Q^\times$ such that
$
\lambda_\prim=c\cdot[\P^\frac{n}{2}]_\prim,
$
and so $V_\lambda=V_{[\P^\frac{n}{2}]}$.
\end{prop}

\begin{proof}
Since $\P^\frac{n}{2}\subseteq X$, we can write 
$$
F=L_1Q_1+\cdots+L_{\frac{n}{2}+1}Q_{\frac{n}{2}+1}.
$$
Choosing derivations $D_1,\ldots,D_{\frac{n}{2}+1}$ such that $D_i(L_j)=\delta_{ij}$ for every $i=1,\ldots,\frac{n}{2}+1$ we get
$$
Q_i=D_i(F)-\sum_{k=1}^{\frac{n}{2}+1}L_kD_i(Q_k)\in J^{F,\lambda}.
$$
In consequence (see Example \ref{dan} for the first equality)
$$
J^{F,[\P^\frac{n}{2}]}=\langle L_1,Q_1,\ldots,L_{\frac{n}{2}+1},Q_{\frac{n}{2}+1}\rangle\subseteq J^{F,\lambda},
$$
hence $J^{F,[\P^\frac{n}{2}]}=J^{F,\lambda}$ and the result follows by Corollary \ref{igualjotas}.
\end{proof}

\begin{proof}\textbf{of Theorem \ref{thm1}} Let $\Sigma$ be a component of $\HL_{n,d}$ passing through the Fermat variety $0\in \HL_{n,d}$, i.e. $X_0=\{F=0\}$ for $F=x_0^d+\cdots+x_{n+1}^d$. Then there exists some $\lambda\in H^{\frac{n}{2},\frac{n}{2}}(X_0,\Z)$ such that $V_\lambda$ is a component of $(\Sigma,0)$. By Proposition \ref{tan} it follows that
$$
\codim_T \ \Sigma=\codim_T \ V_\lambda\ge \codim_{T_0T} \ T_0V_\lambda=\dim_\C \ R^{F,\lambda}_d.
$$
Consider any monomial order and take $x^\alpha\in \C[x_0,\ldots,x_{n+1}]_\sigma\setminus \langle \LT(J^{F,\lambda})\rangle_\sigma$ for 
$$
\sigma=(d-2)(\frac{n}{2}+1)=\soc(J^{F,\lambda}).
$$
Since $J^F\subseteq J^{F,\lambda}$, it follows that $x^\alpha\notin J^F$, and so $x^\alpha|x_0^{d-2}\cdots x_{n+1}^{d-2}$. Applying Proposition \ref{prop6} we get the desired bound
$$
\codim_T \ \Sigma\ge {\frac{n}{2}+d\choose d}-(\frac{n}{2}+1)^2. 
$$
And if the equality holds, it follows from Proposition \ref{prop6} that there exist $L_1,\ldots,L_{\frac{n}{2}+1}\in J^{F,\lambda}_1$ linearly independent. By Theorem \ref{thm1.1} it follows that 
$$
\P^\frac{n}{2}:=\{L_1=L_2=\cdots=L_{\frac{n}{2}+1}=0\}\subseteq X,
$$
and so, by Proposition \ref{prop7}, $\lambda_\prim=c\cdot[\P^{\frac{n}{2}}]_\prim$ for some $\P^\frac{n}{2}\subseteq X_0$ and $c\in\Q^\times$. Therefore every component of $(\Sigma,0)$ is of the form $V_{[\P^\frac{n}{2}]}$ for some $\P^\frac{n}{2}\subseteq X_0$, i.e. 
$$
(\Sigma,0)= (\Sigma_{(1,\ldots,1)},0),
$$
where $\Sigma_{(1,\ldots,1)}$ corresponds to the locus of smooth hypersurfaces containing a linear subvariety of dimension $\frac{n}{2}$, which is a component of $\HL_{n,d}$ by Dan's theorem (Example \ref{dan}).
\end{proof}

\section{Proof of Theorem \ref{thm1.1}}
\label{sec4.2}
In this section we prove Theorem \ref{thm1.1}, which is at the core of the proof of Theorem \ref{thm1} (see \S \ref{sec4.1}). As explained in \S \ref{sec2} (Corollary \ref{igualjotas}) once we have a good enough description of the Artin Gorenstein ideal $J^{F,\lambda}$, we would expect to being able of recovering the cohomological class $\lambda_\prim$ from this ideal. This is exactly what we do. In fact, using the linear equations $L_1,\ldots,L_{\frac{n}{2}+1}\in J^{F,\lambda}_1$ we will compute an explicit expression for a polynomial $P_\lambda\in R^F$ such that
$$
\lambda_\prim=\res\left(\frac{P_\lambda\Omega}{F^{\frac{n}{2}+1}}\right)^{\frac{n}{2},\frac{n}{2}},
$$
see Proposition \ref{prop10}. Once we have computed this polynomial, we will be able to apply the machinery of \cite{VillaPCIAC} to compute the periods of $\lambda$ over all linear cycles of the Fermat variety in terms of the coefficients of $P_\lambda$ (which in turn is expressed in terms of the coefficients of $L_1,\ldots,L_{\frac{n}{2}+1}$). This will give us a family of constrains in the coefficients of each linear polynomial $L_i$, which will be enough to determine completely the coefficients of each $L_i$ (see Proposition \ref{prop11}) and conclude the result.

Let us recall the main tools from \cite{VillaPCIAC} we will use to compute the periods over algebraic cycles. The following proposition tells us how to compute the periods of $\lambda$ over any Hodge cycle $\mu$, once we know the associated polynomials $P_\lambda,P_\mu\in R^F$. Since both classes are rational, the resulting period must be a rational number.

\begin{prop}
\label{prop8}
Let $X=\{F=0\}\subseteq\P^{n+1}$ be a smooth degree $d$ hypersurface of even dimension $n$, and $\lambda,\mu\in H^{\frac{n}{2},\frac{n}{2}}(X,\Q)$ be two Hodge cycles. Let $P_\lambda,P_\mu\in\C[x_0,\ldots,x_{n+1}]_\sigma$ be two homogeneous polynomials of degree $\sigma:=(d-2)(\frac{n}{2}+1)$ such that $\lambda_\prim=\res\left(\frac{P_\lambda\Omega}{F^{\frac{n}{2}+1}}\right)^{\frac{n}{2},\frac{n}{2}}$ and $\mu_\prim=\res\left(\frac{P_\mu\Omega}{F^{\frac{n}{2}+1}}\right)^{\frac{n}{2},\frac{n}{2}}$. Then 
$$
P_\lambda\cdot P_\mu\equiv c\cdot det(Hess(F))\hspace{3mm}\text{(mod }J^F)
$$
for some $c\in \Q$.
\end{prop}

\begin{proof}
Since $\lambda$ and $\mu$ are rational classes, their primitive parts also are rational and so their intersection is a rational number which can be computed via periods as
\begin{equation}
\label{int}
\frac{1}{(2\pi\sqrt{-1})^n}\int_X \res\left(\frac{P_\lambda\Omega}{F^{\frac{n}{2}+1}}\right)\wedge \res\left(\frac{P_\mu\Omega}{F^{\frac{n}{2}+1}}\right)=\frac{-1}{(\frac{n}{2}!)^2}c\cdot (d-1)^{n+2}d.
\end{equation}
Equation \eqref{int} is proved in \cite[\S 6, Proposition 6.1]{VillaPCIAC}. Therefore $c\in\Q$.
\end{proof}

The following proposition taken from \cite{VillaPCIAC} computes the associated polynomial to every linear cycle of the Fermat variety.

\begin{prop}{\normalfont{(\cite[Corollary 8.3]{VillaPCIAC})}}
\label{prop9}
Let
$
X=\{x_0^d+\cdots+x_{n+1}^d=0\}
$ 
be the Fermat variety. For $\alpha_0,\alpha_2,\ldots,\alpha_n\in \{1,3,\ldots,2d-1\}$ consider 
\begin{equation}
\label{lincycle}
\P_{\alpha}^{\frac{n}{2}}:=\{x_{0}-\zeta_{2d}^{\alpha_0}x_{1}=\cdots=x_{n}-\zeta_{2d}^{\alpha_n}x_{{n+1}}=0\},    
\end{equation}
and $\delta:=[\P_{\alpha}^\frac{n}{2}]$. Its associated polynomial is 
\begin{equation}
\label{plinfermat}    
P_{\delta}=c_\delta\cdot\zeta_{2d}^{\alpha_0+\cdots+\alpha_n}\prod_{j=1}^{\frac{n}{2}+1}\left(\frac{x_{{2j-2}}^{d-1}-(\zeta_{2d}^{\alpha_{2j-2}}x_{{2j-1}})^{d-1}}{x_{{2j-2}}-\zeta_{2d}^{\alpha_{2j-2}}x_{{2j-1}}}\right),
\end{equation}
for some $c_\delta\in \Q^\times$.
\end{prop}

The next step is to find an expression for the polynomial $P_\lambda$ in terms of the equations of $L_1,\ldots,L_{\frac{n}{2}+1}\in J^{F,\lambda}_1$. 

\begin{prop}
\label{prop10}
Let $F=x_0^d+\cdots+x_{n+1}^d$ and $X=\{F=0\}$ be the Fermat variety of even dimension $n$ and degree $d\ge 2+\frac{6}{n}$. Let $\lambda\in H^{\frac{n}{2},\frac{n}{2}}(X,\Z)$ be a non-trivial Hodge cycle such that there exist $L_1,\ldots,L_{\frac{n}{2}+1}\in J^{F,\lambda}_1$ linearly independent. Then there exist $c_\lambda\in\C^\times$ and $a_0,a_2,a_4,\ldots,a_n\in\C$ such that up to a permutation of the coordinates
\begin{equation}
\label{plambda}    
P_\lambda=c_\lambda\prod_{j=1}^{\frac{n}{2}+1}\left(\frac{x_{{2j-2}}^{d-1}-(a_{2j-2}x_{{2j-1}})^{d-1}}{x_{{2j-2}}-a_{2j-2}x_{{2j-1}}}\right).
\end{equation}
\end{prop}

\begin{proof}
Since $F=x_0^d+\cdots+x_{n+1}^d$, then $J^F=\langle x_0^{d-1},\ldots,x_{n+1}^{d-1}\rangle\subseteq J^{F,\lambda}$. Without loss of generality we can assume that the set $L_1,\ldots,L_{\frac{n}{2}+1}$ is reduced with respect to the usual lexicographic order, and so modulo a permutation of the variables we have that 
$$
L_i=x_{2i-2}+l_i(x_1,x_3,\ldots,x_{n+1}),
$$
for every $i=1,\ldots,\frac{n}{2}+1$ (note that since we are in Fermat we can permute the variables). In consequence
$$
I:=\langle L_1,\ldots,L_{\frac{n}{2}+1},x_1^{d-1},x_3^{d-1},\ldots,x_{n+1}^{d-1}\rangle\subseteq J^{F,\lambda}.
$$
By Macaulay's theorem, $R:=\C[x_0,\ldots,x_{n+1}]/I$ is Artinian Gorenstein of socle 
$$
\soc(R)=(d-2)(\frac{n}{2}+1)=\sigma=\soc(R^{F,\lambda}).
$$
Then 
$$
J^{F,\lambda}=\langle L_1,\ldots,L_{\frac{n}{2}+1},x_1^{d-1},x_3^{d-1},\ldots,x_{n+1}^{d-1}\rangle.
$$
Considering the lexicographic order generated by the following order in the variables $x_0>x_2>\cdots>x_n>x_1>x_3>\cdots>x_{n+1}$, it follows that as $\C$-vector spaces
$$
J^{F,\lambda}_e\simeq \langle x_0,x_2,\ldots,x_n,x_1^{d-1},x_3^{d-1},\ldots,x_{n+1}^{d-1}\rangle_e \hspace{5mm}\forall e\ge 0,
$$
then by Proposition \ref{hilb} it follows that $L_1,\ldots,L_{\frac{n}{2}+1},x_1^{d-1},\ldots,x_{n+1}^{d-1}$ are a Gr\"obner basis of $J^{F,\lambda}$.
By expanding the following binomial we see that
$$
(L_i-x_{2i-2})^{d-1}\in J^{F,\lambda},
$$
but $L_i-x_{2i-2}=l_i$ depends only on the odd variables $x_1,x_3,\ldots,x_{n+1}$. Therefore, dividing $(L_i-x_{2i-2})^{d-1}$ by the Gr\"obner basis (see Proposition \ref{grob}) we conclude that 
$$
(L_i-x_{2i-2})^{d-1}=a_{i,1}x_1^{d-1}+a_{i,3}x_3^{d-1}+\cdots+a_{i,n+1}x_{n+1}^{d-1},
$$
for some $a_{i,j}\in\C$. The only possibility of this to happen is that $$L_i=x_{2i-2}-a_{2i-2}x_{p_{2i-2}},$$ for some $p_{2i-2}\in \{1,3,\ldots,n+1\}$ and $a_{2i-2}\in\C$. We claim that we can choose these numbers in such a way that the map
$$
2i-2\in\{0,2,\ldots,n\}\mapsto p_{2i-2}\in\{1,3,\ldots,n+1\}
$$
be a bijection. In fact, we have to show that for every $i\neq j$ such that $a_{2i-2},a_{2j-2}\in\C^\times$, then $p_{2i-2}\neq p_{2j-2}$. In order to do this we will compute the associated polynomial 
$$
P_\lambda\in \C[x_0,\ldots,x_{n+1}]_\sigma\setminus J^F_\sigma.
$$
We can assume that 
\begin{equation}
\label{pdelta}
P_\lambda=\sum_{i_0,\ldots,i_{n+1}=0}^{d-2}c_{(i_0,\ldots,i_{n+1})}x_0^{i_0}\cdots x_{n+1}^{i_{n+1}}.
\end{equation}
By Proposition \ref{midef} we have that $L_iP_\lambda\in J^F=\langle x_0^{d-1},\ldots,x_{n+1}^{d-1}\rangle$. Rewriting \eqref{pdelta} as 
$$
P_\lambda=\sum_{j,k=0}^{d-2}q_{j,k}x_{2i-2}^jx_{p_{2i-2}}^k,
$$
for $q_{j,k}$ not depending on $x_{2i-2}$ and $x_{p_{2i-2}}$ for every $j,k\in\{0,\ldots,d-2\}$. We obtain that necessarily $q_{j,k}=a_{2i-2}q_{j+1,k-1}=\cdots=a_{2i-2}^kq_{j+k,0}$ and so $q_{j,k}=0$ for $j+k>d-2$. On the other hand if $l<d-2$ we see that $q_{l,0}=0$ (otherwise $L_iP_\lambda$ contains the non-zero monomial $q_{l,0}x_{2i-2}^{l+1}\notin J^F$). Then
$$
P_\lambda=\sum_{j+k=d-2}q_{d-2,0}x_{2i-2}^j(a_{2i-2}x_{p_{2i-2}})^k=q_{d-2,0}\cdot S_i.
$$
Since $S_i=\prod_{j=1}^{d-2}(x_{2i-2}-\zeta_{d-1}^ja_{2i-2}x_{p_{2i-2}})$ for some $(d-1)$-th primitive root of unity $\zeta_{d-1}$, the $S_i$ are two by two coprimes for $i=1,\ldots,\frac{n}{2}+1$. Therefore, letting $c_\lambda:=q_{d-2,0}$ we have that
$$
P_\lambda=c_\lambda\prod_{i=1}^{\frac{n}{2}+1}\left(\frac{x_{2i-2}^{d-1}-(a_{2i-2}x_{p_{2i-2}})^{d-1}}{L_i}\right).
$$
In consequence $c_\lambda\neq 0$ and if $a_{2i-2},a_{2j-2}\in\C^\times$ then $p_{2i-2}\neq p_{2j-2}$ as claimed (otherwise we would contradict \eqref{pdelta}). In conclusion, we have shown that up to a permutation of the coordinates $P_\lambda$ is given by \eqref{plambda}.
\end{proof}

\begin{rmk}
\label{redad}
It follows that up to a permutation of the coordinates  $$J^{F,\lambda}=(J^F:P_\lambda)=\langle x_{0}-a_0x_{1},x_2-a_2x_3,\ldots,x_{n}-a_nx_{{n+1}},x_{1}^{d-1},x_{3}^{d-1},\ldots,x_{{n+1}}^{d-1}\rangle$$
and so $\{L_1=\cdots=L_{\frac{n}{2}+1}=0\}=\{x_{0}-a_0x_{1}=\cdots=x_{n}-a_nx_{{n+1}}=0\}.$
Therefore, in order to finish the proof of Theorem \ref{thm1.1} we have to show that
\begin{center}
$a_i^d+1=0$ for all $i=0,2,\ldots,n$.
\end{center}
This will follow from a collection of arithmetic conditions which are summarized in the following elementary proposition.
\end{rmk}

\begin{prop}
\label{prop11}
Consider $d$ such that $\zeta_d+\zeta_d^{-1}\notin \Q$ (i.e. $d\neq 1,2,3,4,6$). Suppose $a\in \C$ satisfies that
\begin{equation}
\label{arthcond1}    
\frac{(a^{d-1}+x)(ay-1)}{(a^{d-1}+y)(ax-1)}\in \Q
\end{equation}
for all $x,y\in \{\zeta_{2d}, \zeta_{2d}^3,\zeta_{2d}^5,\ldots,\zeta_{2d}^{2d-1}\}$ such that \eqref{arthcond1} is well defined. Then $a^d+1=0$.
\end{prop}

\begin{proof}
%Note first that $a\neq 0$, otherwise taking $x=\zeta_{2d}^3$ and $y=\zeta_{2d}$ we would have $\zeta_d\in\Q$ which is absurd. 
Assume by contradiction that $a^d+1\neq 0$, i.e. that $a\notin\{\zeta_{2d}, \zeta_{2d}^3,\zeta_{2d}^5,\ldots,\zeta_{2d}^{2d-1}\}$. Therefore \eqref{arthcond1} will be well defined (and non-zero) for all pairs $x,y\in \{\zeta_{2d}, \zeta_{2d}^3,\zeta_{2d}^5,\ldots,\zeta_{2d}^{2d-1}\}\setminus\{-a^{d-1}\}$. For such a pair we also have
$$
\frac{(a^{d-1}+x)(ay-1)}{(a^{d-1}+y)(ax-1)}-1=\frac{(a^{d}+1)(y-x)}{(a^{d-1}+y)(ax-1)}\in\Q^\times.
$$
Since $d\ge 5$ we can take $x_1,x_2,x_3,x_4\in \{\zeta_{2d}, \zeta_{2d}^3,\zeta_{2d}^5,\ldots,\zeta_{2d}^{2d-1}\}\setminus\{-a^{d-1}\}$ all different, then
$$
\varphi(x_1,x_2,x_3,x_4):=\frac{(x_1-x_2)(x_3-x_4)}{(x_1-x_4)(x_3-x_2)}=\frac{\frac{(a^{d}+1)(x_1-x_2)}{(a^{d-1}+x_1)(ax_2-1)}}{\frac{(a^{d}+1)(x_1-x_4)}{(a^{d-1}+x_1)(ax_4-1)}}\cdot\frac{\frac{(a^{d}+1)(x_3-x_4)}{(a^{d-1}+x_3)(ax_4-1)}}{\frac{(a^{d}+1)(x_3-x_2)}{(a^{d-1}+x_3)(ax_2-1)}}\in\Q^\times.
$$
There exists some $i$ such that we can choose them of the form $x_1=\zeta_{2d}^{i+1}$, $x_2=\zeta_{2d}^{i+3}$, $x_3=\zeta_{2d}^{i+5}$, $x_4=\zeta_{2d}^{i+7}$ and so
$$
\varphi(x_1,x_2,x_3,x_4)^{-1}=\frac{1+\zeta_{2d}^2+\zeta_{2d}^4}{-\zeta_{2d}^2}=-1-(\zeta_d+\zeta_d^{-1})\in \Q.
$$
Contradicting our hypothesis on $d$. 

\end{proof}

\begin{proof}\textbf{of Theorem \ref{thm1.1}} After Proposition \ref{prop10} we can assume that $P_\lambda$ has the expression \eqref{plambda}. Thus by Remark \ref{redad}, the proof of Theorem \ref{thm1.1} is reduced to show that $a_i^d+1=0$ for all $i=0,2,4,\ldots,n$.
By Proposition \ref{prop8} and Proposition \ref{prop9}, for every collection $\alpha_0,\alpha_2,\ldots,\alpha_n\in\{1,3,5,\ldots,2d-1\}$ and $\delta:=[\P^\frac{n}{2}_\alpha]$ we have
$$
P_\lambda\cdot P_\delta\equiv c_\alpha\cdot (x_0\cdots x_{n+1})^{d-2} \hspace{3mm}\text{(mod }\langle x_0^{d-1},\ldots,x_{n+1}^{d-1}\rangle)
$$
for some $c_\alpha\in\Q$. Using the expressions \eqref{plambda} and \eqref{plinfermat} we can compute 
$$
c_\alpha=c_\lambda  c_\delta\cdot \zeta_{2d}^{\alpha_0+\cdots+\alpha_n}\prod_{j=1}^{\frac{n}{2}+1}\left(\frac{a_{{2j-2}}^{d-1}-(\zeta_{2d}^{\alpha_{2j-2}})^{d-1}}{a_{{2j-2}}-\zeta_{2d}^{\alpha_{2j-2}}}\right)\in \Q.
$$
For every $i\in\{0,2,\ldots,n\}$ fix the values of $\alpha_j$ for $j\in\{0,2,\ldots,n\}\setminus\{i\}$ in such a way that
$$
\frac{a_j^{d-1}-(\zeta_{2d}^{\alpha_j})^{d-1}}{a_j-\zeta_{2d}^{\alpha_j}}\neq 0.
$$
Now taking $\beta_i:=2d-r$, $\beta'_i:=2d-s$, $\beta_j:=\alpha_j=:\beta'_j$ for $j\neq i$, it follows that 
$$
\frac{c_\beta}{c_{\beta'}}=\frac{(a_i^{d-1}+\zeta_{2d}^r)(a_i\zeta_{2d}^s-1)}{(a_i^{d-1}+\zeta_{2d}^s)(a_i\zeta_{2d}^r-1)}\in\Q
$$
for all $r,s\in\{1,3,\ldots,2d-1\}$ where $c_{\beta'}\neq 0$, and the result follows from Proposition \ref{prop11}.
\end{proof}

\section{Special cohomology classes of the Fermat variety of degree 3, 4 and 6}
\label{sec5.5}
In this section we want to explain why our proof of Theorem \ref{thm1.1} is not enough for covering the cases $d=3,4,6$. In fact, in each case we can find a $(\frac{n}{2}+1)$-parameter family (over $\Q$) of $(\frac{n}{2},\frac{n}{2})$ cohomology classes of the Fermat variety $X=\{x_0^d+\cdots+x_{n+1}^d=0\}$, all different in primitive cohomology, satisfying \eqref{igualdad1} and intersecting every linear cycle in a rational number. Each family comes with a natural faithful $(\mathbb{S}^1_{\Q(\zeta_d)})^{\frac{n}{2}+1}$ action, where $$\mathbb{S}^1_{\Q(\zeta_d)}:=\{z\in\Q(\zeta_d)^\times: |z|=1\}<\C^\times.$$ 

The reason behind the existence of these families is the failure of Proposition \ref{prop11}. In fact, for every $d=3,4,6$ the following sets $G_d$ are constituted by elements satisfying \eqref{arthcond1} as can be easily verified
$$
G_3:=\mathbb{S}^1_{\Q(\zeta_3)},\hspace{4mm}    
G_4:=\zeta_8\cdot \mathbb{S}^1_{\Q(i)},\hspace{4mm}    
G_6:=i\cdot \mathbb{S}^1_{\Q(\zeta_3)}. 
$$
For every $a=(a_0,a_2,\ldots,a_n)\in G_d^{\frac{n}{2}+1}$ define 
\begin{equation}
\label{lambdaa}    
\lambda_a:=\res\left(\frac{P_a\Omega}{F^{\frac{n}{2}+1}}\right)^{\frac{n}{2},\frac{n}{2}}\in H^{\frac{n}{2},\frac{n}{2}}(X)_\prim
\end{equation}
given by the polynomial  
\begin{equation}
\label{pol}
P_a:=c_a\prod_{i=1}^{\frac{n}{2}+1}\left(\frac{x_{2i-2}^{d-1}-(a_{2i-2}x_{2i-1})^{d-1}}{x_{2i-2}-a_{2i-2}x_{2i-1}}\right)    
\end{equation} 
where $c_a\in\C^\times$ will be defined later. Let $\delta:=[\P^\frac{n}{2}_\alpha]$ be the linear cycle given by \eqref{lincycle} whose associated polynomial $P_\delta$ is given by \eqref{plinfermat}. By Proposition \ref{prop8} the intersection of $\lambda$ with $\delta$ is equal (up to multiplication by a non-zero rational number) to
\begin{equation}
\label{intnumb}
c_a\cdot c_\alpha:=c_a  c_\delta\cdot \zeta_{2d}^{\alpha_0+\cdots+\alpha_n}\prod_{i=1}^{\frac{n}{2}+1}\left(\frac{a_{{2i-2}}^{d-1}-(\zeta_{2d}^{\alpha_{2i-2}})^{d-1}}{a_{{2i-2}}-\zeta_{2d}^{\alpha_{2i-2}}}\right).
\end{equation}
Fixing an $\widetilde\alpha$ such that $c_{\widetilde\alpha}\neq 0$, let us define $c_a:=c_{\widetilde\alpha}^{-1}$. Now it is clear that \eqref{intnumb} is rational for any $\alpha$, since every $a_{2i-2}$ satisfies \eqref{arthcond1}. With a little more effort it is possible to show that in fact the intersection of $\lambda_a$ with all linear cycles contained in the Fermat variety (which can be obtained from the $\P^\frac{n}{2}_\alpha$'s after permuting the coordinates of $\P^{n+1}$) is also a rational number. Since 
$$
(J^{F}:P_a)=\langle x_0-a_0x_1,\ldots,x_n-a_nx_{n+1},x_1^{d-1},\ldots,x_{n+1}^{d-1}\rangle
$$
it follows as in the proof of Corollary \ref{igualjotas} that every pair $\lambda_a$, $\lambda_b$ with $a\neq b$ are $\C$-linearly independent, and so they are all different.

\section{Proof of Theorem \ref{thm2}}
\label{sec5}
Let $0\in \HL_{n,d}$ be the Fermat variety, i.e. $X_0=\{F=0\}$ for $F=x_0^d+\cdots+x_{n+1}^d$. Let $0\in\Delta\subseteq T$ be a polydisc and $\lambda\in \Gamma(\Delta,R^n\pi_*\Z)$ be such that $0\in V_\lambda$ but \eqref{igualdad1} does not hold. Let us prove the inequality \eqref{desteo2} for $d\ge \text{max}\{2+\frac{6}{n},4\}$. Consider any monomial order and take $x^\alpha\in \C[x_0,\ldots,x_{n+1}]_\sigma\setminus \langle \LT(J^{J,\lambda})\rangle_\sigma$ for 
$
\sigma=(d-2)(\frac{n}{2}+1)=\soc(J^{F,\lambda}).
$
It follows from Proposition \ref{hilb} and Remark \ref{comp} that 
\begin{equation}
\label{dessalfa2}
\dim_\C \ R^{F,\lambda}_d\ge \#S^d_\alpha.
\end{equation}
\noindent\textbf{Case 1:} If $x^\alpha$ is not of the form \begin{equation}
\label{monmin2}    
x^\alpha=x_{i_1}^{d-2}\cdots x_{i_{\frac{n}{2}+1}}^{d-2}
\end{equation}
for some $0\le i_1<\cdots<i_{\frac{n}{2}+1}\le n+1$, then \eqref{desteo2} follows from \eqref{dessalfa2} and Proposition \ref{prophos}. 

\bigskip

\noindent\textbf{Case 2:} If $x^\alpha$ is of the form \eqref{monmin2} we can assume $x^\alpha=x_1^{d-2}x_3^{d-2}\cdots x_{n+1}^{d-2}$ after some relabeling of the variables. Since \eqref{dessalfa2} is strict in this case, there must exist some $x^\beta\in \C[x_0,\ldots,x_{n+1}]_d\setminus (\langle\LT(J^{F,\lambda})\rangle_d\cup S^d_\alpha)$. In consequence there exists some $i$ even such that $x_i|x^\beta$ and so $x_i\in \C[x_0,\ldots,x_{n+1}]_1\setminus (\langle\LT(J^{F,\lambda})\rangle_1\cup S^1_\alpha)$, i.e. $\dim_\C \ R^{F,\lambda}_1>\#S^1_\alpha$. By duality $\dim_\C \ R^{F,\lambda}_{\sigma-1}>\#S^{\sigma-1}_\alpha$ and so there exists some $x^{\alpha'}\in \C[x_0,\ldots,x_{n+1}]_{\sigma-1}\setminus (\langle\LT(J^{F,\lambda})\rangle_{\sigma-1}\cup S^{\sigma-1}_\alpha)$. By Proposition \ref{hilb} it follows that 
$$
\dim_\C \ R^{F,\lambda}_d\ge \#(S^d_\alpha\cup S^d_{\alpha'})=\#S^d_\alpha+\#(S^d_{\alpha'}\setminus S^d_\gamma)
$$
where $\gamma_i=\text{min}\{\alpha_i,\alpha'_i\}$ for every $i=0,\ldots,n+1$. Let us divide the analysis into two cases: 

\bigskip

\noindent\textbf{Case 2.1:} If 
$$
\#\{i\in\{0,\ldots,n+1\}:\alpha'_i=d-2\}=\frac{n}{2},
$$
then by Proposition \ref{prophos} $\#S^d_{\alpha'}\ge \#S^d_{\alpha_0}$ for $\alpha_0:=(0,\ldots,0,1,d-4,d-2,\ldots,d-2)$. Also in this case, up to some relabeling of the variables we can assume $x^\gamma|x_1^{d-4}(x_3x_5\cdots x_{n+1})^{d-2}$, and so $\#S^d_\gamma\le \#S^d_{\gamma_0}$ for $\gamma_0:=(0,\ldots,0,d-4,d-2,\ldots,d-2)$ with $\text{deg}(x^{\gamma_0})=\sigma-2$. Therefore
$$
\text{dim}_\C\text{ }R^{F,\lambda}_d\ge \#S^d_\alpha+\#S^d_{\alpha_0}-\#S^d_{\gamma_0}={\frac{n}{2}+d\choose d}+{\frac{n}{2}+d-1\choose d-1}-\left(\frac{3n^2}{8}+\frac{9n}{4}+2\right).
$$

\noindent\textbf{Case 2.2:} If
$$
\#\{i\in\{0,\ldots,n+1\}:\alpha'_i=d-2\}<\frac{n}{2}.
$$
Again by Proposition \ref{prophos} we have $\#S^d_\alpha\ge \#S^d_{\alpha_1}$ for $\alpha_1:=(0,\ldots,0,1,d-3,d-3,d-2,\ldots,d-2)$, and up to relabeling of the variables we have that $x^\gamma|(x_1x_3)^{d-3}(x_5x_7\cdots x_{n+1})^{d-2}$, and so $\#S^d_\gamma\le \#S^d_{\gamma_1}$ for $\gamma_1:=(0,\ldots,0,d-3,d-3,d-2,\ldots,d-2)$ with $\text{deg}(x^{\gamma_1})=\sigma-2$. Then 
$$
\text{dim}_\C\text{ }R^{F,\lambda}_d\ge \#S^d_\alpha+\#S^d_{\alpha_1}-\#S^d_{\gamma_1}={\frac{n}{2}+d\choose d}+{\frac{n}{2}+d-1\choose d-1}-\left(\frac{n^2}{4}+\frac{5n}{2}+2\right)
$$
$$
\ge {\frac{n}{2}+d\choose d}+{\frac{n}{2}+d-1\choose d-1}-\left(\frac{3n^2}{8}+\frac{9n}{4}+2\right)
$$
as desired.

\bigskip

%\noindent\textbf{Case 2:} Assume now that $d=3\ge 2+\frac{6}{n}$. 

%\bigskip

%We can assume up to relabeling that $x^\alpha=x_1x_3\cdots x_{n+1}$. By a similar argument as in Case 2 we can assume that there exists $x^{\alpha'}=x_0x_5x_7\cdots x_{n+1}\in\C[x_0,\ldots,x_{n+1}]_{\sigma-1}\setminus (\langle\LT(J^{F,\lambda})\rangle_{\sigma-1}\cup S^{\sigma-1}_\alpha)$. By Proposition \ref{hilb} we get the desired bound $$\text{dim}_\C\text{ }R^{F,\lambda}_d\ge \#(S^d_\alpha\cup S^d_{\alpha'})={\frac{n}{2}+1\choose 3}+{\frac{n}{2}-1\choose 2}.$$

Let us analyze now the equality in \eqref{desteo2} for $n\ge 4$ (the case $n=2, d\ge 5$ was established by \cite{voisin89}): When the equality of \eqref{desteo2} holds, by a similar argument as in the equality of \eqref{igualdad1} (see the proof of Proposition \ref{prop6}) in each of the three cases considered above we can compute $\langle\LT(J^{F,\lambda})\rangle_k$ for $k\ge d$ and corresponds respectively in Case 1, Case 2.1 and Case 2.2 (up to some relabeling of the variables) to
\begin{equation}
\langle\LT(J^{F,\lambda})\rangle_k=\langle x_0,x_2,\ldots,x_{n-2},x_n^2,x_1^{d-1},x_3^{d-1},\ldots,x_{n-1}^{d-1},x_{n+1}^{d-2}\rangle_k,    
\end{equation}
\begin{equation}
\langle\LT(J^{F,\lambda})\rangle_k=\langle x_0,x_2,\ldots,x_{n-2},x_n^2,x_1^{d-1},x_3^{d-1},\ldots,x_{n-1}^{d-1},x_{n+1}^{d-1},x_nx_{n+1}^{d-3}\rangle_k,   \end{equation}
\begin{equation}
\label{descartar}
\langle\LT(J^{F,\lambda})\rangle_k=\langle x_0,x_2,\ldots,x_{n-2},x_n^2,x_1^{d-1},x_3^{d-1},\ldots,x_{n-1}^{d-1},x_{n+1}^{d-1},x_nx_{n-1}^{d-2},x_nx_{n+1}^{d-2}\rangle_k.    
\end{equation}
By duality, these equalities are also valid for $k\le \sigma-d$. The hypothesis $n,d\ge 4$ implies that $d-2\le\sigma-d$ (also note that for $n\ge 4$, the inequality \eqref{desteo2} is strict in the Case 2.2, then we do not consider the case \eqref{descartar}) thus in any case  $\forall  k=0,\ldots,d-2$
\begin{equation}
\label{eqfin1.1}    
J^{F,\lambda}_k=\langle L_0,L_2,\ldots,L_{n-2},C_n,D_{n+1}\rangle_k,
\end{equation}
for some homogeneous polynomials with $\LT(L_i)=x_i$, for $i=0,2,\ldots,n-2$, $\LT(C_n)=x_n^2$ and $deg(D_{n+1})=d-2$. Therefore there exists a complete intersection of type $(1,1,\ldots,1,2)$ given by $Z:=\{L_0=\cdots =L_{n-2}=C_n=0\}$ such that $I(Z)\subseteq J^{F,\lambda}$. And if we assume further that $F\in (J^{F,\lambda})^2$ and $d\ge 5$, then 
$$
F=f_1g_1+\cdots+f_mg_m
$$
for some homogeneous polynomials $f_i,g_i\in J^{F,\lambda}$. Without loss of generality we can assume that $deg(f_i)\le deg(g_i)$, then $deg(f_i)\le d-3$ and so $f_i\in \langle L_0,L_2,\ldots,L_{n-2},C_n\rangle$ for every $i=1,\ldots,m$, then we can rewrite
$$
F=L_0P_0+L_2P_2+\cdots+L_{n-2}P_{n-2}+C_nQ_n
$$
for $P_0,\ldots,P_{n-2},Q_n\in J^{F,\lambda}$. This implies that $Z\subseteq X_0$ and also that
$$
J^{F,[Z]}=\langle L_0,L_2,\ldots,L_{n-2},C_n,P_0,\ldots,P_{n-2},Q_n\rangle\subseteq J^{F,\lambda}.
$$
Since both ideals are Artinian Gorenstein of the same socle it follows that $J^{F,[Z]}=J^{F,\lambda}$, and by Corollary \ref{igualjotas} we get that $\lambda_\prim=a[Z]_\prim$ for some $a\in \Q^\times$.

\section{Final comments}
\label{sec6}
Despite of Theorem \ref{thm1}, Movasati's conjecture remains open for $d=3,4,6\ge 2+\frac{6}{n}$, and we believe it is also true in these cases. We describe two possible ways to prove this. The first possibility is to follow our method of exploiting the rationality of the Hodge cycle $\lambda$. After \cite{VillaPCIAC} we know how to compute the intersection of $\lambda_\prim$ with other known algebraic cycles inside the Fermat variety. But this method does not tell us how to intersect $\lambda_\prim$ with all rational cohomology classes. We know that all the primitive part of the de Rham cohomology of $X=\{F=0\}\subseteq\P^{n+1}$ is given by Griffiths' basis, i.e. is generated by elements of the form
$$
\res\left(\frac{P\Omega}{F^{q+1}}\right)\in F^{n-q}H^n_\dR(X)_\prim.
$$
And so, even if we generalize Proposition \ref{prop8} to any pair of forms of Griffiths' basis, it not clear at all which combinations of these elements are rational classes. Dually, instead of computing the intersection products, we would try to compute the periods over rational or integral homology cycles. This approach would work at least computationally since we can construct a basis for $H_n(X,\Z)$ in terms of the \textit{vanishing cycles} (also called \textit{Pham cycles}) of $U:=X\setminus\P^n_\infty$ and the \textit{cycle at infinity} $Z_\infty=X\cap\P^n_\infty$. The vanishing cycles are given explicitly and it is possible to compute all the periods of Griffiths' basis over them, see \cite[Lemma 7.12]{dmos} and \cite[Proposition 15.1]{ho13}. The problem is that this method only gives us the periods of the whole 
$$
\omega_\lambda:=\res\left(\frac{P_\lambda\Omega}{F^{\frac{n}{2}+1}}\right)\in F^\frac{n}{2}H^n_\dR(X)_\prim
$$
which are not necessarily rational numbers (the numbers that are rational are the periods of its $(\frac{n}{2},\frac{n}{2})$ piece). Therefore, we still need to know how to determine explicitly some $P\in\C[x_0,\ldots,x_{n+1}]_{d\frac{n}{2}-n-2}$ (in terms of $P_\lambda$) such that 
$$
\omega_\lambda-\omega_\lambda^{\frac{n}{2},\frac{n}{2}}=\res\left(\frac{P\Omega}{F^\frac{n}{2}}\right)\in F^{\frac{n}{2}+1}H^n_\dR(X)_\prim.
$$
The second method is trying to use Voisin's geometric argument \cite[\S 3]{voisin1988}. In order to do this, we need at least to show that $\text{dim}_\C\text{ }J^{F,\lambda}_1$ is locally constant at the Fermat point through every local Hodge loci $V_\lambda$ satisfying \eqref{igualdad1}. After Theorem 1, this ended up being the case for $d\neq 3,4,6$. Furthermore it turns out that the whole Hilbert function of $R^{F,\lambda}$ is locally constant at Fermat. If we show this continuity of the Hilbert function of $R^{F,\lambda}$ at least for degrees 1 and 2, we would be able to characterize the components attaining the equality in \eqref{desteo2} as the one corresponding to a complete intersection of type $(1,\ldots,1,2)$. We want to point out that Voisin's argument never uses the rationality of $\lambda$, and so if we take a class $\mu\in H^{\frac{n}{2},\frac{n}{2}}(X)_\prim$ and consider its associated $F^\frac{n}{2}$-locus
$$
V_\mu^{\frac{n}{2}}=\{t\in T: \mu_t\in F^\frac{n}{2}H^n_\dR(X_t)\},
$$
where $\mu_t$ is the unique flat deformation of $\mu$ with respect to the Gauss-Manin connection, then all classes $\mu$ satisfying \eqref{igualdad1} (e.g. the classes described in \S \ref{sec5.5}) with $\text{dim}_\C\text{ }J^{F,\mu}_1$ locally constant are rational and so Hodge. Since we can construct easily $(\frac{n}{2},\frac{n}{2})$ classes with non rational primitive part $\mu$ satisfying \eqref{igualdad1} (just define it as \eqref{lambdaa} with random values $a_i\in\C$ for $i=0,2,\ldots,n$) we see that the continuity of the Hilbert function of $R^{F,\mu}$ fails in general for any $F^\frac{n}{2}$-locus, and so the continuity of this function (even for some degrees) is strongly related to the Hodge condition. In view of the variational Hodge conjecture we can ask ourselves if the Hilbert function of $R^{F,[Z]}$ is continuous for any flat family $Z\rightarrow T$ of algebraic cycles $[Z_t]\in \CH^\frac{n}{2}(X_t)$. This is clear in certain cases such as families of complete intersection cycles, but it is not clear to the author in the general case.

%\bigskip
%\newpage
%\bibliography{biblio}
%\bibliographystyle{alpha}

\end{document}